\documentclass[11pt]{amsart}
\usepackage{amssymb}
\usepackage{amssymb, amsmath}
\usepackage{verbatim}
\usepackage{color}
\let\d=\partial

\let\wt=\widetilde
\let\wh=\widehat

\def\cC{{\mathcal C}}
\def\cD{{\mathcal D}}

\def\cF{{\mathcal F}}

\def\cI{{\mathcal I}}

\def\cL{{\mathcal L}}

\def\cQ{{\mathcal Q}}

\def\cX{{\mathcal X}}

\def\N{{\mathbb N}}
\def\R{{\mathbb R}}

\def\Z{{\mathbb Z}}

\def\virgp{\raise 2pt\hbox{,}}
\def\cdotpv{\raise 2pt\hbox{;}}

\def\div{ \hbox{\rm div}\,  }

\def\ddj{\dot \Delta_j}

\def\da{\delta\!a}
\def\dv{\delta\!v}

\def\dh{\delta\!h}
\def\dK{\delta\!K}
\def\dL{\delta\!L}

\newtheorem{thm}{Theorem}[section]

\newtheorem{rmk}{Remark}[section]

\newtheorem{prop}{Proposition}[section]

\newcommand{\ben}{\begin{eqnarray}}
\newcommand{\een}{\end{eqnarray}}
\newcommand{\beno}{\begin{eqnarray*}}
\newcommand{\eeno}{\end{eqnarray*}}

\newcommand{\with}{\quad\!\hbox{with}\!\quad}
\newcommand{\andf}{\quad\!\hbox{and}\!\quad}

\title[One-dimensional Navier-Stokes equations]{The  one-dimensional compressible
Navier-Stokes equations in critical regularity spaces}
\author{Rapha\"el Danchin\\ Universit\'e Paris-Est Cr\'eteil}

\address{Universit\'{e} Paris-Est,  LAMA (UMR 8050), UPEMLV, UPEC, CNRS, Institut Universitaire de France,
 61 avenue du G\'{e}n\'{e}ral de Gaulle, 94010 Cr\'{e}teil Cedex, France.}
 \email{raphael.danchin@u-pec.fr}

\keywords{Compressible Navier-Stokes, one-dimensional, critical regularity, large time behavior, high viscosity limit}
\subjclass[2020]{76N10, 35Q30}

\begin{document}

\begin{abstract}  We are concerned with the barotropic compressible Navier-Stokes equations on the real line.
Our primary goal is to establish the global well-posedness in a critical regularity framework 
in the case where the initial data are small perturbations of a stable constant state. 
Surprisingly, even though the result in the multi-dimensional case is by now classical,
 the one-dimensional case has not been elucidated yet as far as we know. 
This is due to the fact that in the critical framework, the regularity of the velocity is so negative that
 some nonlinear terms are  out of control.  
Here, we overcome the difficulty by considering the equations in the mass Lagrangian coordinates system. 

Granted with a global well-posedness statement, we then establish optimal time decay estimates
and  investigate the high viscosity limit, pointing out the convergence of the specific volume to the solution
of some ordinary differential equation, after  time and space rescaling.
\end{abstract}
\maketitle

We consider  the barotropic compressible Navier-Stokes equations on the real line: 
\begin{equation}\label{eq:NSC}\left\{
 \begin{array}{l}
  \rho_t +(\rho u)_x =0,\\[5pt]
(\rho u)_t + (\rho u^2)_x -( \mu u_x)_x + P_x=0,
 \end{array}\right.
\end{equation}
where $u=u(t,x)\in\R$ and $\rho=\rho(t,x)\in\R_+$ (with $t\in\R_+$ and $x\in\R$) stand for the velocity
and the density of the fluid, respectively. 
We  assume that   the viscosity coefficient $\mu$ is a smooth positive function of $\rho,$ and that  the pressure 
$P$ is a smooth function of $\rho.$ 
Equations \eqref{eq:NSC} are supplemented with initial data $(\rho_0,u_0)$ at time $t=0.$
\medbreak

The study of the initial value problem associated to \eqref{eq:NSC} has a long history. 
In a pioneering work \cite{Kanel}  Kanel'   proved the global existence of classical finite energy solutions 
 in the case where the initial density is bounded and bounded away from zero and goes to some positive constant at infinity.
 There, the equations are written in the mass Lagrangian coordinates system that will be presented  below.

 In a subsequent work \cite{KS}, Kazhikhov and Shelukhin obtained the  short-time existence  and uniqueness of $H^1$ 
 (or smoother) solutions for the full one-dimensional compressible Navier-Stokes equations governing the evolution of perfect gases, 
 when both the density and temperature
are bounded and bounded away from zero.
Those solutions have been shown to be global (for appropriate choice of pressure functions)
a bit later by Kawashima and Nishida in \cite{KN}. 

Another approach consisting in looking for less regular solutions, possibly discontinuous, has been initiated by Serre in \cite{Serre} and 
Hoff in \cite{Hoff87}. The Lagrangian coordinates version of \eqref{eq:NSC} 
is considered and the initial velocity is only in  $L^2$ while the initial density, still bounded away from zero,
is  in the space $BV$ of functions with bounded variations (or in the Sobolev space $H^s$ with $s>0$ in Serre's work). The solutions therein are  local in time. 
Global existence has been proved later by Hoff in \cite{Hoff98} for  strictly increasing pressure laws. In his work, 
the density still has to be bounded and bounded away from zero, but  the $BV$ assumption is no longer needed. 
In fact, it is only required that $(\rho_0-\underline\rho,u_0-\underline u)$ is in $L^2$ where 
$(\underline\rho,\underline u)$ stands for any pair of smooth monotonous functions connecting different constant states at $-\infty$ and $+\infty.$ 
In these  works, the quantity $\mu u_x -P$ which is sometimes referred to as
the `viscous effective flux'  plays a fundamental role. 
Very recently, Chen, Ha, Hu and Nguyen established in  \cite{CHHN}  a global existence and uniqueness result
in the case of small data for any density with  $BV$ regularity and bounded away from zero, and 
velocity with some positive Sobolev regularity. 

The common point between the aforementioned works is that vacuum is not allowed and  that
 the viscosity coefficient is constant (although this latter property is 
not so fundamental).
A third approach based on the relative entropy method 
has been 
used by  Mellet and Vasseur in \cite{MV}.  There,   $(\rho_0-\underline\rho,u_0-\underline u)$  are in $H^1,$ and
the viscosity coefficient is allowed to depend on the density and even to degenerate near vacuum. 
Uniqueness is shown if the viscosity is bounded away from zero. One can finally mention 
the  recent paper by Tan, Wang and Zhang \cite{TWZ} where vacuum is permitted
and exponential decay results are proved when the equations are posed on a segment with appropriate boundary conditions. 
\medbreak
Our goal here is to prove that  the compressible Navier-Stokes equations are 
well-posed in the so-called critical regularity framework, like in the multi-dimensional case. 
This critical regularity approach originates from the seminal work by Fujita and Kato  \cite{FK} on the incompressible Navier-Stokes equations.  In our context, it
is based on the observation  that if one neglects the (lower order) pressure term, then the Cauchy problem for  \eqref{eq:NSC} is invariant for all positive 
real number $\lambda$ by the rescaling:
$$(\rho,u)(t,x)\leadsto (\rho, \lambda u)(\lambda^2t,\lambda x)\andf 
(\rho_0,u_0)(t,x)\leadsto (\rho_0, \lambda u_0)(\lambda x).$$
It is thus expected that ``optimal'' functional settings for proving well-posedness by means of the contractive mapping theorem 
have the above invariance.
\smallbreak
Let us  focus on solutions such that
\begin{equation*}
(\rho,u)(t,x)\to (\bar\rho, 0)\quad\hbox{as }\ |x|\to\infty,\quad \hbox{for some }\ \bar\rho\in\R_+.
\end{equation*}
In the homogeneous Besov spaces setting,  critical regularity corresponds to choosing initial data $(\rho_0,u_0)$ such that
\begin{equation}\label{eq:critical}(\rho_0-\bar\rho) \in \dot B^{d/p}_{p,1}\andf u_0\in \dot B^{-1+d/p}_{p,1},\end{equation}
and, indeed, it has been pointed out in a number of works (see e.g. 
 \cite{CD,CMZ,D1,Haspot11})
 that the compressible Navier-Stokes equations in dimension $d\geq2$ 
 are locally well-posed if \eqref{eq:critical} is satisfied (with  $p$ not too large and $\rho_0$ bounded away from zero), and  globally well-posed if 
 $(\rho_0,u_0)$ is close enough to $(\bar\rho,0)$ and $P'(\bar\rho)>0$.  
 For proving such results,  one usually rewrites  the compressible Navier-Stokes equations
under  nonconservative form after dividing the velocity equation by the density. In 
the one-dimensional setting, this amounts  to considering 
\begin{equation*}
\left\{ \begin{array}{l}
  \rho_t + \rho u_x+ \rho u_x =0,\\[5pt]
 u_t +  uu_x - \rho^{-1}( \mu u_x)_x + \rho^{-1}P_x=0.
 \end{array}\right.
\end{equation*}
However,  if $d=1$ then \eqref{eq:critical} makes 
 the regularity
of the  velocity so negative that  one cannot control  nonlinear terms like $\rho^{-1}( \mu u_x)_x$
by means of standard product laws. Typically,  the  product  
is not defined  on $\dot B^{1/p}_{p,1}\times \dot B^{1/p-1}_{p,1}$ if $p>2$, and has range in 
too big a space to be  compensated by  the parabolic smoothing of the velocity equation, if $p\leq2.$

This difficulty disappears if we manage to `put one derivative outside the nonlinear terms'.
To do so, rather than  considering \eqref{eq:NSC} under it original form, we rewrite it 
 in the \emph{mass Lagrangian coordinates system}, namely, following   the presentation of \cite{AKM}, 
we make the change of variables: 
\begin{equation}\label{eq:lag}
\Phi: (t,x)\mapsto (t,y)\with {\rm dy}=\rho\,{\rm  dx}-\rho u\,{\rm dt},\end{equation}
then  look at the governing equations for:
\begin{equation}
v(t,y):= u(t,x)\andf \eta(t,y):=\rho^{-1}(t,x)\with (t,y):=\Phi(t,x).\end{equation}
 This change of variable is justified if, say, the functions $\rho$ and $u$ are Lipschitz with, additionally, $\rho>0.$
  The  Jacobian matrix is
 $$
 D\Phi=\begin{pmatrix} 1&0\\
 -\rho u&\rho\end{pmatrix}$$
 and the inverse change of variable is 
 \begin{equation}\label{eq:lag-1}  \Phi^{-1}:(t,y)\mapsto (t,x)\with {\rm dx}= \eta\,{\rm dy}+v\,{\rm dt}.\end{equation}
 Since $\partial_x$ translates into $\eta^{-1}\partial_y$ and 
 $$(\eta^{-1}  v_t) (t,y)= \bigl((\rho u)_t+(\rho u^2)_x\bigr)(t,x),$$
  the equations corresponding to \eqref{eq:NSC} are 
  \begin{equation}\label{eq:NSClag}\left\{ \begin{array}{l}
  \eta_t - v_y =0,\\[5pt]
v_t + Q_y -( \nu v_y)_y=0
 \end{array}\right.\end{equation}
with $Q(\eta)=P(\eta^{-1})$ and $\nu(\eta)=\eta^{-1}\mu(\eta^{-1}).$
\medbreak
The present work is devoted to studying the Cauchy  problem and the asymptotic behavior of global solutions of \eqref{eq:NSClag}
 in the aforementioned  critical regularity setting, for initial data $(\eta_0,v_0)$ close to a steady state 
 $(\bar\eta,0)$ with $\bar\eta>0$ which is linearly stable in the sense that  
 \begin{equation}\label{eq:stab} 
Q'(\bar\eta)<0.
\end{equation}



\section{Results} \label{s:results}

Before stating our results,  introducing some notation and our functional framework is in order. 
We fix once and for all a homogeneous Littlewood-Paley decomposition $(\ddj)_{j\in\Z}$  (see its construction 
in e.g. \cite[Chap. 2]{BCD}) and define for all $p\in[1,\infty],$ $r\in[1,\infty]$ and $\sigma\in\R$ 
the following homogeneous Besov (semi)-norm:
$$\|z\|_{\dot B^\sigma_{p,r}}:= \Bigl\| 2^{j\sigma}\|\ddj z\|_{L^p(\R)}\Bigr\|_{\ell^r(\Z)}.$$ 
The homogeneous Besov space $\dot B^\sigma_{p,r}$ is the set  of tempered distributions $z$
such that 
$$\|z\|_{\dot B^\sigma_{p,r}}<\infty\andf \lim_{\lambda\to\infty} \|\chi(\lambda D) z\|_{L^\infty}=0,$$ 
 for some (or, equivalently, all) $\chi\in \cC^\infty_c(\R)$ with $\chi(0)\not=0.$
This definition guarantees that $(\dot B^\sigma_{p,r};\|\cdot\|_{\dot B^\sigma_{p,r}})$ is a normed space. 
In this paper, we will mostly have  $r=1,$ $p\in[1,\infty]$ and $\sigma\leq1/p,$
which ensures completeness of  $\dot B^\sigma_{p,1}.$ 
\smallbreak
An important ingredient of our analysis is to allow for different regularity of low and high frequencies. 
This motivates to introduce  for all $\alpha\in\R_+,$ $p\in[1,\infty]$ and $\sigma\in\R$ the notation
\begin{equation}\label{eq:not1}
\|z\|_{\dot B^\sigma_{p,1}}^{\ell,\alpha}:= \sum_{j\leq j_0+1+\log_2 \alpha} 2^{j\sigma}\|\ddj z\|_{L^p}\andf
\|z\|_{\dot B^\sigma_{p,1}}^{h,\alpha}:= \sum_{j\geq j_0+\log_2 \alpha} 2^{j\sigma}\|\ddj z\|_{L^p}.
\end{equation}
In the same spirit, we set
\begin{equation}\label{eq:not2}
z^{\ell,\alpha}:= \sum_{j\leq j_0+\log_2 \alpha} \ddj z\andf
z^{h,\alpha}:= \sum_{j>j_0+\log_2 \alpha} \ddj z.
\end{equation}
The value of the  cut-off parameter $j_0\in\Z$  will be universal: it will be determined during the proof of our global existence
theorem. 

The small overlap between low and high frequencies in \eqref{eq:not1} is intentional: it will enable 
us to use (repeatedly) the following
two inequalities\footnote{Throughout the paper,  the standard notation $A\lesssim B$ means that there exists
 a positive number $C$ whose exact value is irrelevant such that $A\leq CB.$}:
\begin{equation}\label{eq:lh}
\|z^{\ell,\alpha}\|_{\dot B^\sigma_{p,1}}\lesssim \|z\|_{\dot B^\sigma_{p,1}}^{\ell,\alpha}\andf
\|z^{h,\alpha}\|_{\dot B^\sigma_{p,1}}\lesssim  \|z\|_{\dot B^\sigma_{p,1}}^{h,\alpha}.
\end{equation}
As in many recent works dedicated to equations of parabolic type, 
in order to take advantage of  optimal maximal regularity results, it is convenient to use the following norms:
\begin{equation}\label{eq:tilde}
\|z\|_{\wt L^m_t(\dot B^\sigma_{p,1})}:= \sum_{j\in\Z} 2^{j\sigma}\|\ddj z\|_{L^m(0,t;L^p(\R))},\quad
1\leq m,p\leq\infty,\quad\!  0\leq t\leq\infty,\quad\! \sigma\in\R
\end{equation}
which, owing to Minkowski inequality, may be compared with the more classical 
norms in $L^m(0,t;\dot B^s_{p,1})$ as follows: 
\begin{equation}\label{eq:Mink}
\|z\|_{L^m_t(\dot B^\sigma_{p,1})}:= \|z\|_{L^m(0,t;\dot B^\sigma_{p,1})}\leq 
\|z\|_{\wt L^m_t(\dot B^\sigma_{p,1})}.\end{equation}
Let us first state a local well-posedness result for general data with critical regularity 
and nonvanishing specific volume.
\begin{thm}\label{thm:main0} Fix some reference specific volume $\bar\eta>0$ and
smooth pressure function $Q.$ 
Consider any data 
$(\eta_0,v_0)$  such that $a_0:=\eta_0-\bar\eta$  (resp. $v_0$) belongs to $\dot B^{1/p}_{p,1}$
(resp. $\dot B^{-1+1/p}_{p,1}$) for some $p\in[1,\infty).$ If, in addition,   $\inf_x \eta_0(x)>0$
then there exists $T>0$ such that Equations \eqref{eq:NSClag} have a unique solution $(\eta,v)$ with 
$$\inf_{(t,x)\in[0,T]\times\R} \eta(t,x)>0,\quad\!
(\eta-\bar\eta)\in \cC([0,T];\dot B^{1/p}_{p,1})\!\!\andf\!\!
v\in \cC([0,T];\dot B^{-1+1/p}_{p,1})\times L^1(0,T;\dot B^{1+1/p}_{p,1}).$$
\end{thm}
The proof goes along the lines of that of  the multi-dimensional case in \cite{D6}. 
It is in fact much less technical since convection terms are absent. The details are left to the reader. 
\medbreak
The main part of the article consists in proving the following global existence statement: 
\begin{thm} \label{thm:main1}  
Take any  reference specific volume $\bar\eta>0.$ Assume that $Q$  satisfies 
\eqref{eq:stab} and $Q(\bar\eta)=0$ (with no loss of generality). Let $\bar\nu:=\nu(\bar\eta),$  ${\rm Ma}:=1/\sqrt{-Q'(\bar\eta)}$ and $\check\nu:={\rm Ma}\,\bar\nu.$
\smallbreak
 There exists a constant $c$  depending only on the functions $\check\nu$ and $\check Q$ defined by
 \begin{equation}\label{eq:renormalized}
\check Q(\wt\eta):= {\rm Ma}^{2} \bar\eta^{-1} Q(\bar\eta\wt\eta)\andf \check\nu(\wt\eta):=\bar\nu^{-1}\nu(\bar\eta\wt\eta)\end{equation}
 such that 
for any initial data $v_0$ and $\eta_0=\bar\eta+a_0$ satisfying  for some $p\in[2,4],$ 
\begin{equation}\label{eq:smalldata}
 {\rm Ma}^{-1}\|a_0\|_{\dot B^{-1/2}_{2,1}}^{\ell,\check\nu^{-1}}+\bar\nu\|a_0\|_{\dot B^{1/p}_{p,1}}^{h,\check\nu^{-1}}+
\|v_0\|_{\dot B^{-1/2}_{2,1}}^{\ell,\check\nu^{-1}}+\|v_0\|_{\dot B^{-1+1/p}_{p,1}}^{h,\check\nu^{-1}} \leq c\bar\eta\bar\nu\end{equation}
Equations   \eqref{eq:NSClag} with initial data $(\eta_0,v_0)$ have a unique global solution $(\eta=\bar\eta+a,v)$ such that
$$\displaylines{ (a^{\ell,\check\nu^{-1}},v^{\ell,\check\nu^{-1}})\in \cC(\R_+; \dot B^{-1/2}_{2,1})\cap L^1(\R_+;\dot B^{3/2}_{2,1}), 
\cr
a^{h,\check\nu^{-1}}\in \cC(\R_+; \dot B^{1/p}_{p,1})\cap L^1(\R_+;\dot B^{1/p}_{p,1})\andf
v^{h,\check\nu^{-1}}\in \cC(\R_+; \dot B^{-1+1/p}_{p,1})\cap L^1(\R_+;\dot B^{1+1/p}_{p,1}).}$$
Furthermore, there exists an absolute constant $C$ such that for all $t\in\R_+,$ we have
\begin{multline}\label{eq:final}
{\rm Ma}^{-1}\|a(t)\|_{\dot B^{-1/2}_{2,1}}^{\ell,\check\nu^{-1}} 
+\bar\nu \|a(t)\|_{\dot B^{1/p}_{p,1}}^{h,\check\nu^{-1}}  +  \|v(t)\|_{\dot B^{-1/2}_{2,1}}^{\ell,\check\nu^{-1}} 
 +\|v(t)\|_{\dot B^{-1+1/p}_{p,1}}^{h,\check\nu^{-1}}  \\
+\int_0^t \Bigl({\rm Ma}^{-2}\bigl(\check\nu\|a\|^{\ell,\check\nu^{-1}}_{\dot B^{3/2}_{2,1}} + \|a\|^{h,\check\nu^{-1}}_{\dot B^{1/2}_{2,1}}\bigr)+\bar\nu\|v\|_{\dot B^{3/2}_{2,1}}^{\ell,\check\nu^{-1}}
+\bar\nu\|v\|_{\dot B^{1+1/p}_{p,1}}^{h,\check\nu^{-1}}+\|v_t\|^{h,\check\nu^{-1}}_{\dot B^{-1+1/p}_{p,1}}\Bigr)\\
\lesssim {\rm Ma}^{-1}\|a_0\|_{\dot B^{-1/2}_{2,1}}^{\ell,\check\nu^{-1}} 
+\bar\nu \|a_0\|_{\dot B^{1/p}_{p,1}}^{h,\check\nu^{-1}}  +  \|v_0\|_{\dot B^{-1/2}_{2,1}}^{\ell,\check\nu^{-1}} 
 +\|v_0\|_{\dot B^{-1+1/p}_{p,1}}^{h,\check\nu^{-1}}.\end{multline}
 \end{thm}
\begin{rmk} The restriction $p\leq 4$ comes from the fact that the high frequencies of the data
 belong to $L^p$-type Besov spaces while low frequencies have to be bounded in $L^2$-type spaces.
  As in \cite{CD,CMZ}, larger values of $p$ might  be considered
if the high frequencies of the data also belong to   Besov spaces of type $\dot B^\sigma_{2,1}$  
with \emph{supercritical} regularity. 

Very recently, Guo, Song and Yang  pointed out in \cite{GSY} that  for the
three-dimensional compressible Navier-Stokes equations,  one can  do without  $L^2$-type regularity 
in low frequencies, if it is replaced by a suitable supercritical regularity assumption 
in spaces of type $L^q$.  
Although  we suspect this to be also possible in the one-dimensional case, we chose not to pursue this direction to avoid supplementary technicalities.
\end{rmk}
Our second result states optimal algebraic rates of convergence for $(\eta,u)(t)$ to $(\bar\eta,0)$ 
when~$t$ goes to infinity.  To avoid technicalities, we only consider the case $p=2$ (as in \cite{DX}, 
it is expected that similar estimates hold true for the solutions given by Theorem \ref{thm:main1}
with  $p>2$), and for the sake of conciseness, we assume  (with no loss of generality, see the beginning of the next section)
that $\bar\eta={\rm Ma}=\bar\nu=1.$
\begin{thm} \label{thm:main2}  Let $(\eta=1+a,u)$ be a solution given by Theorem \ref{thm:main1}, 
under the smallness condition \eqref{eq:smalldata}.
Set  $\langle t\rangle:=\sqrt{1+t^2}$ and 
$$\cX_{2,0}=\|a_0\|_{\dot B^{-1/2}_{2,1}}
+\|a_0\|_{\dot B^{1/2}_{2,1}}  +  \|v_0\|_{\dot B^{-1/2}_{2,1}}.$$
 Then we have for all $t\geq0,$
\begin{align}\label{eq:decay}
 \cD(t)\leq C\cX_{2,0}\  &\hbox{ where }\  \cD(t):=\cD^\ell(t)+\cD^h_a(t)+\wt\cD^h_v(t)\\
 \with &\cD^\ell(t):=\|\langle \tau\rangle (a,v)\|_{\wt L^\infty_t(\dot B^{3/2}_{2,1})}^{\ell,1}
 + \|\langle\tau\rangle(a,v)\|^{\ell,1}_{\wt L^2_t(\dot B^{5/2}_{2,1})},\nonumber\\
   &\cD^h_a(t):=\|\langle \tau\rangle^{3/2} a\|_{\wt L^\infty_t(\dot B^{1/2}_{2,1})}^{h,1}\andf
  \wt\cD^h_v(t):=\|\tau \langle \tau\rangle^{1/2} v\|_{\wt L^\infty_t(\dot B^{3/2}_{2,1})}^{h,1}.\nonumber
  \end{align}
\end{thm}

\begin{rmk}   For general data with critical regularity, we believe the low frequency decay 
given by  $\cD^\ell$ to be  optimal, inasmuch as it is the one of the corresponding
free heat equation. For example, one can deduce from \eqref{eq:decay} and interpolation with \eqref{eq:final} that
$$\|(a,v)(t)\|_{L^2}\leq Ct^{-1/2} \cX_{2,0},$$ which is the optimal inequality for the heat equation
 supplemented with data in $\dot B^{-1/2}_{2,1}$. 
 As established   in  \cite{BSXZ} by Brandolese, Shou, Xu and Zhang,
 this rate is also optimal for the multi-dimensional Navier-Stokes equations supplemented 
 with  `generic' data.  
We expect a similar result for \eqref{eq:NSC}, but demonstrating this would considerably 
lengthen the article.

As for  high frequency decay,  the exponent $\tau^{3/2}$ is sharp with our method (and improves  the rate $d/2+1-\varepsilon$ obtained in \cite{DX} in dimension $d\geq2$).  It is given by  the quadratic term $((a^\ell)^2)_y$
coming from the pressure term  that cannot have faster decay 
than $\langle t\rangle^{-3/2}.$ 
\end{rmk}
\begin{rmk} It is worth comparing our results with those obtained by K. Chen et al in \cite{CHHN}
in the constant viscosity case.
There, $\eta_0$ only has to be in the $BV$ space,  and can thus be discontinuous at a countable number of points.
In contrast, since the space $\dot B^{1/2}_{2,1}$ is continuously embedded in the space of continuous functions, 
our initial specific volume has to be continuous. 
As for the velocity, in \cite{CHHN} it is required  to have \emph{positive} regularity (namely to be in the space $W^{s,1}$ for 
some $s>0$) while, here,  it can be in a space with \emph{negative} regularity 
like $\dot B^{-1/2}_{2,1}$ in low frequencies, and even in $\dot B^{-3/4}_{4,1}$ in high  frequencies.
Note that \eqref{eq:decay} allows to recover the same decay rate $t^{-1/2}$ for $\|(a,v)(t)\|_{L^\infty}$ as 
 in \cite{CHHN}. It  encodes much more information, though.
\end{rmk}

The last part of the paper is dedicated to  studying  the behavior of the global 
solutions $(\eta,v)$ given by Theorem \ref{thm:main1} when the viscosity tends to infinity, an asymptotic that does not seem to have been studied much to date. 
It has to be noted that Inequality \eqref{eq:final}  just implies boundedness of the solution in some
functional space. However,  the  more accurate analysis based on the explicit computations
for the linearized equations recalled in the Appendix, hint that $v\to0$ while $a\to a_0$ when 
$\bar\nu$ goes to $\infty.$  

To exhibit  more accurate asymptotics,  let us perform a  `diffusive rescaling', namely,
\begin{equation}\label{eq:diffrescaling}(\eta,v)(t,y)=(\check \eta,\bar\nu^{-1} \check v)(\bar \nu^{-1} t,y).\end{equation}
The equations for $(\check \eta,\check v)$  read: 
\begin{equation}\label{eq:NL3}\left\{\begin{array}{l}
  \check \eta_t - \check v_y =0,\\[3pt]
\bar\nu^{-2}\check v_t +  (Q(\check \eta))_y - \bar\nu^{-1}( \nu(\check \eta) \check v_y)_y=0,\\[3pt]
(\check\eta,\check v)|_{t=0}=(\eta_0,\bar\nu v_0).
 \end{array}\right.
\end{equation}
We thus expect  $\bar\nu^{-1}\nu(\check \eta) \check v_y- Q(\check \eta)$ to tend to $0$ when $\bar\nu$ goes to $\infty.$
Plugging this information in  the first equation  of \eqref{eq:NL3} gives  
that   $\check \eta$   tends to the solution $\check \theta$ of the ordinary differential equation
\begin{equation}\label{eq:limit1}
\check \theta_t-\bar\nu(\nu^{-1}Q)(\check \theta)=0,
\end{equation}
and thus, eventually, 
\begin{equation}\label{eq:limit2}
\check v_y \to \bar\nu(\nu^{-1}Q)(\check \theta).\end{equation}
The last part of the paper is devoted to justifying this heuristics. Again,  to simplify the presentation, we
restrict ourselves to the functional setting of Theorem \ref{thm:main1} with  $p=2.$
\begin{thm}\label{thm:main3}
There exists a universal constant $c>0$ such that if 
$a_0:=\eta_0-\bar\eta$ and $v_0$ satisfy
\begin{equation}\label{eq:smallnessnu}\check\nu^{-1}\|a_0\|_{\dot B^{-1/2}_{2,1}} + \|a_0\|_{\dot B^{1/2}_{2,1}}  
 + \bar\nu^{-1} \|v_0\|_{\dot B^{-1/2}_{2,1}}\leq c\bar\eta,\end{equation}
 then Equations \eqref{eq:NL3} have a unique global solution $(\check \eta=\bar\eta +\check a,\check v)$ satisfying for all $t\in\R_+,$
 \begin{multline}\label{eq:finalbis}
 \check\nu^{-1}\|\check a(t)\|_{\dot B^{-1/2}_{2,1}}+ \|\check a(t)\|_{\dot B^{1/2}_{2,1}}
 +  \bar\nu^{-2}\|\check v(t)\|_{\dot B^{-1/2}_{2,1}}
+ {\rm Ma}^{-2}\int_0^t \bigl(\check\nu\|\check a\|^{\ell,\check\nu^{-1}}_{\dot B^{3/2}_{2,1}} + \|\check a\|^{h,\check\nu^{-1}}_{\dot B^{1/2}_{2,1}}
\bigr)\\
+\int_0^t \bigl(\bar\nu^{-2}\|\check v_t\|^{h,\check\nu^{-1}}_{\dot B^{-1/2}_{2,1}}+\|\check v_{y}\|_{\dot B^{1/2}_{2,1}}\bigr)
\lesssim \check\nu^{-1}\|a_0\|_{\dot B^{-1/2}_{2,1}} + \|a_0\|_{\dot B^{1/2}_{2,1}} 
 +  \bar\nu^{-1}\|v_0\|_{\dot B^{-1/2}_{2,1}}.\end{multline}
 If, moreover, $(a_0,v_0)$ is bounded   in the space 
 $\dot B^{-\sigma}_{2,1}$ for some $\sigma\in(1/2,3/2),$   then 
 $\check \eta$ converges uniformly on $\R_+\times \R$ to the global solution  
 $$\check\theta\in \biggl(\bar\eta +\cC_b(\R_+;\dot B^{1/2}_{2,1})\cap L^1(\R_+;\dot B^{1/2}_{2,1})\biggr)$$ 
  of \eqref{eq:limit1} supplemented with 
 initial data $\eta_0,$ with the rate $\check\nu^{-\frac{2\sigma-1}{\sigma+3/2}}\cdotp$
 
 Finally,  \eqref{eq:limit2} holds true strongly in $L^1_{loc}(\R_+;\dot B^{1/2}_{2,1}).$  
  \end{thm}
  The rest of the paper is structured as follows. 
  Our  global well-posedness result for critical regularity data close to a stable equilibrium is proved in Section \ref{s:GWP}. 
Then, optimal decay estimates for the constructed solutions are established in  Section \ref{s:decay}. In the final section, 
 we  prove  that, after performing  the rescaling \eqref{eq:diffrescaling}, the density converges to a solution 
of \eqref{eq:limit1}.  Explicit formulae for the solutions to the linearized compressible Navier-Stokes equations are recalled in Appendix, 
as well as some heuristics on the the high viscosity limit.



\section{The proof of global well-posedness}\label{s:GWP}

In order to reduce the proof to  the case 
\begin{equation}\label{eq:reduction}
\bar\eta=1,\quad Q'(1)=-1\andf  \nu(1)=1,
\end{equation} 
we first 
make the change of unknowns  $\eta=\bar\eta\wt\eta$ and $v=\bar\eta\wt v$  in \eqref{eq:NSClag} 
so that $(\wt\eta,\wt v)$ satisfies 
 \begin{equation}\label{eq:NL2}\left\{\begin{array}{l}
  \wt\eta_t - \wt v_y =0,\\[5pt]\wt v_t + {\rm Ma}^{-2} \check Q_y - \bar \nu ( \check\nu \wt v_y)_y=0\end{array}\right.
\end{equation}
with $\check Q$ and $\check\nu$ defined in \eqref{eq:renormalized}. We note  that $\check Q'(1)=-1$ and $\check\nu(1)=1$.
\medbreak
Then, we set
\begin{equation}\label{eq:change}\begin{aligned}
\wt\eta(t,y)=\check \eta(Tt,Yy)\andf \wt v(t,y)=V \check v (Tt,Yy)\qquad\cr\with 
T:={\rm Ma}^{-2}\bar\nu^{-1},\quad Y:={\rm Ma}^{-1}\bar\nu^{-1}\andf V: ={\rm Ma}^{-1}.\end{aligned}\end{equation} 
In this way, we see that $(\check\eta,\check v)$ satisfies \eqref{eq:NSClag} with 
reference specific volume equal to one,  and pressure and viscosity functions 
 satisfying \eqref{eq:reduction}. The threshold between low and high frequencies is shifted from $\alpha=\check\nu^{-1}$ to  $\alpha=1,$
and  the scaling properties of homogeneous Besov norms 
(see e.g. \cite[Chap. 2]{BCD})
guarantee that  we have for all $\sigma\in\R$ and $p\in[1,\infty],$ 
\begin{equation}\label{eq:rescaling}\begin{aligned}
\|\check\eta(Tt)\|_{\dot B^{1/p+\sigma}_{p,1}}=Y^{-\sigma} \|\wt\eta(t)\|_{\dot B^{1/p+\sigma}_{p,1}},
\quad\int_0^{Tt} \|\check\eta\|_{\dot B^{1/p+\sigma}_{p,1}}=Y^{-\sigma}T \int_0^{t}   \|\wt\eta\|_{\dot B^{1/p+\sigma}_{p,1}}
\\
\|\check v(Tt)\|_{\dot B^{1/p+\sigma}_{p,1}}=V^{-1}Y^{-\sigma} \|\wt v(t)\|_{\dot B^{1/p+\sigma}_{p,1}},
\quad
\int_0^{Tt} \|\check v\|_{\dot B^{1/p+\sigma}_{p,1}}=V^{-1}Y^{-\sigma}T \int_0^{t}   \|\wt v\|_{\dot B^{1/p+\sigma}_{p,1}}\\\andf
\int_0^{Tt} \|\check v_t\|_{\dot B^{1/p+\sigma}_{p,1}}=V^{-1}Y^{-\sigma} \int_0^{t}   \|\wt v_t\|_{\dot B^{1/p+\sigma}_{p,1}}.\qquad\qquad
\end{aligned}\end{equation}
In the rest of this section,  we drop the checks. Then, setting 
$a:=\eta-1,$   Equations  \eqref{eq:NSClag} rewrite
 \begin{equation}\label{eq:NSClaglin}\left\{
 \begin{array}{l}
  a_t - v_y =0,\\[5pt]
v_t -  a_y -v_{yy}=g:= (a K(a))_y+(L(a) v_y)_y
 \end{array}\right.\end{equation}
 for some  smooth functions $K$ and $L$ vanishing at $0,$ that may be computed from $\check Q$ and $\check\nu.$ For notational simplicity, the parameter $\alpha$ (now equal to $1$) coming into play  in 
 \eqref{eq:not1} and \eqref{eq:not2} will be omitted. 

\subsection{A priori estimates}

Here we  prove a priori estimates for \eqref{eq:NSClaglin}  seen as a linear system.
\medbreak
For $p\in[1,\infty],$ $(s,s')\in\R^2$ and $t\in\R_+,$  let us set, using the notation defined in \eqref{eq:tilde}, 
\begin{equation*}
\begin{aligned}
\cX_p^{s,s'}(t)&:=\|(a,v)\|_{\wt L_t^\infty(\dot B^{s}_{2,1})}^\ell\! +\!  \|(a_y,v)\|_{\wt L^\infty_t(\dot B^{s'}_{p,1})}^h
\!+\!\int_0^t\! \bigl(\|(a,v)\|_{\dot B^{s+2}_{2,1}}^\ell \!+\!\|(a_y,v_{yy},v_t)\|_{\dot B^{s'}_{p,1}}^h\bigr)\\\andf
\cX_{p,0}^{s,s'}&:= \|(a_0,v_0)\|_{\dot B^{s}_{2,1}}^\ell +  \|(a_{0,y},v_0)\|_{\dot B^{s'}_{p,1}}^h.\hspace{6cm}\end{aligned}
\end{equation*}
We claim that for any regularity indices $s$ and $s',$ Lebesgue exponent $p\in[1,\infty]$ and time $t\in\R_+,$  we have
\begin{equation}\label{eq:linear}
\cX^{s,s'}_p(t)\lesssim \cX_{p,0}^{s,s'}+\int_0^t\bigl(\|g\|^\ell_{\dot B^s_{2,1}}+\|g\|^h_{\dot B^{s'}_{p,1}}\bigr)\cdotp
\end{equation}
To prove our claim, we localize \eqref{eq:NSClaglin}  by means of $\ddj,$ getting 
 \begin{equation}\label{eq:NSClagj}\left\{ \begin{array}{l}
  a_{j,t} - v_{j,y} =0,\\[5pt]
v_{j,t} - a_{j,y} -v_{j,yy}=g_j, \end{array}\right.\end{equation}
with    $ a_j:=\ddj a,$  $v_j:=\ddj v$  and $g_j:=\ddj g.$

\subsubsection*{Step 1: the low frequencies}

As in \cite{D1}, we look at the evolution of the following functional where $\kappa>0$ will be fixed later: 
 $$\cL_j^2:=\|(a_j, v_j)\|^2_{L^2}-2\kappa\int_\R v_j a_{j,y}.$$
From \eqref{eq:NSClagj}, it is easy to get for all $j\in\Z,$ 
\begin{equation}\label{eq:lf1}
\frac12\frac d{dt}\cL_j^2+\kappa\|a_{j,y}\|_{L^2}^2+(1-\kappa) \|v_{j,y}\|_{L^2}^2=-\kappa 
\int_\R v_{j,yy} a_{j,y} -\kappa \int_\R g_j a_{j,y} + \int_\R g_j v_j.
\end{equation}
Using the spectral localization of $v_j$ and Young inequality, we may write
for some universal constant $C_B,$
\begin{align}\label{eq:lf2}
\biggl|\int_\R  v_{j,yy} \,a_{j,y}\biggr|&\leq  \frac{1}{2} \| a_{j,y}\|_{L^2}^2+ C_B2^{2j} \|v_{j,y}\|_{L^2}^2\\
\andf\label{eq:lf3}\int_\R g_j  a_{j,y} &\leq  C_B2^{j} \|g_j\|_{L^2}\|a_j\|_{L^2}.\end{align}
Hence, plugging \eqref{eq:lf2} and \eqref{eq:lf3} in \eqref{eq:lf1}, we get
\begin{multline}\label{eq:lf4}
\frac12\frac d{dt}\cL_j^2+\frac\kappa2\|a_{j,y}\|_{L^2}^2+(1-\kappa) \|v_{j,y}\|_{L^2}^2\\
\leq C_B2^{2j}\kappa \|v_{j,y}\|_{L^2}^2 + C_B2^j\kappa \|g_j\|_{L^2}\|a_j\|_{L^2}+  \|g_j\|_{L^2}\|v_j\|_{L^2}.
\end{multline}
Note that by Young  inequality and spectral localization,  we have 
$$
\biggl|2\kappa \int v_j a_{j,y}\biggr| \leq  \frac12\|v_j\|_{L^2}^2 +2 C_B^2\kappa^22^{2j}\|a_j\|_{L^2}^2
$$
which in particular ensures 
\begin{equation}\label{eq:equiv}\cL_j\simeq \|(a_j,v_j)\|_{L^2}\end{equation}
 whenever $j\leq j_0$ with $j_0\in\Z$ satisfying
\begin{equation}\label{eq:condj1}
2C_B2^{j_0}\kappa \leq1.
\end{equation}
If, furthermore, 
\begin{equation}\label{eq:condj2}\kappa\biggl(\frac32+C_B2^{2j_0}\biggr) \leq1\end{equation}
then 
$$1-\kappa-C_B2^{2j}\kappa\geq \kappa/2\quad\hbox{for all }\ j\leq j_0.$$
Therefore,  if  both \eqref{eq:condj1} and \eqref{eq:condj2} are satisfied then, reverting to \eqref{eq:lf4} allows to conclude that
there exist two universal positive constants $c_B<1<C_B$ such that  
\begin{equation}\label{eq:lffinal}
\frac12\frac d{dt}\cL_j^2 + c_B2^{2j}\cL_j^2 \leq C_B \|g_j\|_{L^2}\cL_j.
\end{equation}
Hence, integrating, we get by standard arguments: 
$$\cL_j(t)+c_B2^{2j}\int_0^t\cL_j\leq \cL_j(0)+ C_B\int_0^t \|g_j\|_{L^2}.$$
Multiplying the above inequality by $2^{js},$  using \eqref{eq:equiv}, then summing on $j\leq j_0,$ we get for any $s\in\R$ and $t\geq0$
(changing slightly $c_B$ and $C_B$ if necessary):
\begin{equation}\label{eq:lf}
\|(a,v)\|_{\wt L^\infty_t(\dot B^s_{2,1})}^\ell+c_B\int_0^t\|(a, v)\|_{\dot B^{s+2}_{2,1}}^\ell\leq C_B\biggl(\|(a_0,v_0)\|_{\dot B^s_{2,1}}^\ell
+\int_0^t\|g\|_{\dot B^{s}_{2,1}}^\ell\biggr)\cdotp
\end{equation}

\subsubsection*{Step 2: High frequencies}

Adapting to the one-dimensional case  the approach of  Haspot in \cite{Haspot11} (borrowed 
from earlier works by Hoff \cite{Hoff87,Hoff98}), we introduce  the  `effective velocity':
\begin{equation}\label{eq:hf1}
w(t,y):=v(t,y)+\int_0^y a(t,z)\,dz.
\end{equation}
From \eqref{eq:NSClaglin}, we discover that
\begin{equation}\label{eq:hf2} w_t- w_{yy}=g+v
\end{equation}
and it is obvious that 
\begin{equation} \label{eq:hf3} a_t +a = w_y.\end{equation}
On the one hand, leveraging the endpoint parabolic maximal regularity  estimates (see e.g. 
\cite[Chap. 3]{BCD})  to handle the high frequencies of $w$  and the standard Bernstein inequalities
(see \cite[Lemma 2.1]{BCD}) which imply  that 
$$\|v\|^h_{\dot B^{s'}_{p,1}}\lesssim 2^{-2j_0}\|v\|^h_{\dot B^{s'+2}_{p,1}},$$
and the fact that $v_{yy}=w_{yy}-a_y,$ we get for some universal constant $C_0$ and all $t\geq0$:
\begin{align*}
\|w\|^h_{\wt L^\infty_t(\dot B^{s'}_{p,1})}+ \int_0^t \|w\|^h_{\dot B^{s'+2}_{p,1}} &\leq 
C_0\biggl(\|w_0\|^h_{\dot B^{s'}_{p,1}}+\int_0^t \|g\|_{\dot B^{s'}_{p,1}}^h
+ \int_0^t \|v\|^h_{\dot B^{s'}_{p,1}}\biggr)\\
 &\leq C_0\biggl(\|w_0\|^h_{\dot B^{s'}_{p,1}}+\int_0^t \|g\|_{\dot B^{s'}_{p,1}}^h
+ 2^{-2j_0}\int_0^t \|v\|^h_{\dot B^{s'+2}_{p,1}}\biggr)
\\ &\leq C_0\biggl(\|w_0\|^h_{\dot B^{s'}_{p,1}}+\int_0^t \|g\|_{\dot B^{s'}_{p,1}}^h
+ 2^{-2j_0}\int_0^t \bigl( \|w\|^h_{\dot B^{s'+2}_{p,1}}\!+\!\|a\|^h_{\dot B^{s'+1}_{p,1}}\bigr)\biggr)\cdotp
\end{align*}
On the other hand, from \eqref{eq:hf3}, we immediately have  for all $t\geq0,$
$$\|a\|^h_{\wt L^\infty_t(\dot B^{s'+1}_{p,1})}+ \int_0^t\|a\|^h_{\dot B^{s'+1}_{p,1}}
\leq \|a_0\|^h_{\dot B^{s'+1}_{p,1}}+\int_0^t\|w_y\|^h_{\dot B^{s'+1}_{p,1}}.$$
Putting together  the above inequalities and assuming that $j_0$ is large enough 
(to have $2^{-2j_0}$ small enough), we conclude that there exists an absolute constant $C_0$ such that for all $t\geq0,$
\begin{multline}\label{eq:wh}
\|a\|^h_{\wt L^\infty_t(\dot B^{s'+1}_{p,1})}+\|w(t)\|^h_{\wt L^\infty_t(\dot B^{s'}_{p,1})}+  \int_0^t \bigl(\|w\|^h_{\dot B^{s'+2}_{p,1}}+ \|a\|^h_{\dot B^{s'+1}_{p,1}}\bigr)
\\\leq C_0\biggl(\|a_0\|^h_{\dot B^{s'+1}_{p,1}}+\|w_0\|^h_{\dot B^{s'}_{p,1}}+\int_0^t\|g\|^h_{\dot B^{s'}_{p,1}}\biggr)\cdotp
\end{multline}
Owing to \eqref{eq:hf1}, we have 
\begin{equation}\|v- w\|_{\dot B^{s'}_{p,1}}^h\lesssim 2^{-2j_0}\|a\|_{\dot B^{s'+1}_{p,1}}^h
\andf
\|v- w\|_{\dot B^{s'+2}_{p,1}}^h\lesssim \|a\|_{\dot B^{s'+1}_{p,1}}^h.\end{equation}
Hence, $w$ may be replaced with $v$ in \eqref{eq:wh}.  Furthermore, since $v_t=w_{yy}+g,$ 
we can bound $v_t$ in $L^1(0,t;\dot B^{s'+2}_{p,1})$ like $w.$ 
Then, combining with \eqref{eq:lf} completes the proof of  \eqref{eq:linear}.

\subsubsection*{Step 3: Nonlinear estimates}

Let us introduce the notation
\begin{align*} 
\cX_p(t)&:=\!\|(a,v)\|_{\wt L^\infty_t(\dot B^{-1/2}_{2,1})}^{\ell} \!+\!  \|(a_y,v)\|_{\wt L^\infty_t(\dot B^{-1+1/p}_{p,1})}^{h}
\!+\!\int_0^t\!\bigl( \|(a,v)\|_{\dot B^{3/2}_{2,1}}^{\ell} +\|(a_y,v_{yy},v_t)\|_{\dot B^{-1+1/p}_{p,1}}^{h}\bigr)\\
\hbox{and }\ \cX_{p,0}&:=\|(a_0,v_0)\|_{\dot B^{-1/2}_{2,1}}^{\ell} + \|(a_{0,y},v_0)\|_{\dot B^{-1+1/p}_{p,1}}^{h}.
\end{align*}
Taking $s=-1/2$ and $s'=1/p-1$ in \eqref{eq:linear}, we discover that 
\begin{equation}\label{eq:NL1}
\cX_p(t) \lesssim \cX_{p,0}+\int_0^t \|g\|_{\dot B^{-1/2}_{2,1}}^\ell  +\int_0^t  \|g\|_{\dot B^{-1/p+1}_{p,1}}^h.
\end{equation} 
To estimate $g$ (defined in \eqref{eq:NSClaglin}),  we have to remember that  $a^\ell$ and $a^h$ are expected to be small in 
$L^\infty(\R_+;\dot B^{-1/2}_{2,1})$ and $L^\infty(\R_+;\dot B^{1/p}_{p,1}),$ respectively.
Since  all spaces $\dot B^{1/q}_{q,1}$ are continuously embedded in the set~$\cC_b$ of bounded continuous
functions on $\R,$ and as  
\begin{equation}\label{eq:lfhf}
 \|z\|^\ell_{\dot B^{s'}_{2,1}}\lesssim  \|z\|^\ell_{\dot B^{s}_{2,1}}\andf  \|z\|^h_{\dot B^{s}_{p,1}}\lesssim  \|z\|^h_{\dot B^{s'}_{p,1}}
 \quad\hbox{whenever }\  s\leq s', \end{equation}
  the classical results of stability of Besov spaces by left composition give us
\begin{equation}\label{eq:compo}
\|F(a)\|_{\dot B^{\sigma}_{q,1}}\lesssim \|a\|_{\dot B^{\sigma}_{q,1}},\qquad \sigma>0\andf q\in[1,\infty]
\end{equation}
for all smooth function $F$ vanishing at $0.$ 
\medbreak
In  the case $p=2,$  it is only a matter of bounding $g$ in $L^1(\R_+;\dot B^{-1/2}_{2,1}).$ 
To do so, we use \eqref{eq:compo} 
and the fact that  $\partial_y:\dot B^{1/2}_{2,1}\to \dot B^{-1/2}_{2,1}.$ We get
$$\begin{aligned}
\|g\|_{\dot B^{-1/2}_{2,1}}&\lesssim \| aK(a)\|_{\dot B^{1/2}_{2,1}} + \|L(a)v_y\|_{\dot B^{1/2}_{2,1}} \\
&\lesssim \| a\|_{\dot B^{1/2}_{2,1}}^2 + \|a\|_{\dot B^{1/2}_{2,1}}\|v_y\|_{\dot B^{1/2}_{2,1}}.
\end{aligned}
$$
Hence,  since 
$$ \| a\|_{L^2_t(\dot B^{1/2}_{2,1})}^2\lesssim  \| a\|_{L^1_t(\dot B^{3/2}_{2,1})}^\ell \| a\|_{L^\infty_t(\dot B^{-1/2}_{2,1})}^\ell
 + \| a\|_{L^1_t(\dot B^{1/2}_{2,1})}^h \| a\|_{L^\infty_t(\dot B^{1/2}_{2,1})}^h,$$
we end up with $$\int_0^t\|g\|_{\dot B^{-1/2}_{2,1}}\lesssim 
\| a\|_{L^2 _t(\dot B^{1/2}_{2,1})}^2 + \|a\|_{L^\infty_t(\dot B^{1/2}_{2,1})}
\|v_y\|_{L^1_t(\dot B^{1/2}_{2,1})}\lesssim  X_2^2(t),$$
whence
$$\cX_2(t) \leq C\bigl(\cX_2(0)+ \cX_2^2(t)\bigr)\cdotp$$ 
We thus get the desired global-in-time control for small data in the particular case $p=2.$
\smallbreak
To handle the general case $p\in[2,4],$ we have to establish quadratic estimates 
for  the last two terms of \eqref{eq:NL1}.  Toward this, we shall use repeatedly the fact that as a consequence of interpolation
and H\"older inequality, we have for all $r\in[1,\infty],$ 
\begin{equation}\label{eq:interpo1}
 \|(a,v)\|^\ell_{L^r_t(\dot B^{-1/2+2/r}_{2,1})}+\|a\|^h_{L^r_t(\dot B^{1/p}_{p,1})}
 +\|v\|^h_{L^r_t(\dot B^{-1+1/p+2/r}_{p,1})} \lesssim   \cX_p(t)\quad\hbox{for all }\   t>0.\end{equation}
 Now, using  the stability by product and left-composition of the space $\dot B^{1/p}_{p,1}$  and the fact that
\begin{equation}\label{eq:embed}
\|z\|_{\dot B^{1/q+\sigma}_{q,1}}\lesssim \|z\|_{\dot B^{1/2+\sigma}_{2,1}}\quad\hbox{for all }\ q\geq2\andf\sigma\in\R,\end{equation}
we get
$$\begin{aligned}\int_0^t\|(aK(a))_y\|^h_{\dot B^{-1+1/p}_{p,1}}&\lesssim \int_0^t\|aK(a)\|_{\dot B^{1/p}_{p,1}}\\
&\lesssim \int_0^t\|a\|_{\dot B^{1/p}_{p,1}}^2\\&\lesssim 
\int_0^t\Bigl(\|a\|_{\dot B^{1/2}_{2,1}}^\ell+\|a\|_{\dot B^{1/p}_{p,1}}^h\Bigr)^2\lesssim \cX_p^2(t).\end{aligned}$$
Similarly, 
$$\begin{aligned}
\int_0^t\|(L(a)v_y)_y\|^h_{\dot B^{-1+1/p}_{p,1}}&\lesssim \int_0^t\|L(a)v_y\|_{\dot B^{1/p}_{p,1}}\\
&\lesssim \int_0^t\|a\|_{\dot B^{1/p}_{p,1}}\|v\|_{\dot B^{1/p}_{p,1}}\\&\lesssim 
\|a\|_{L_t^2(\dot B^{1/p}_{p,1})}\|v\|_{L^2_t(\dot B^{1/p}_{p,1})}
\lesssim \cX_p^2(t).\end{aligned}$$
Handling the low frequencies is more involved, and requires $p$ to be in $[2,4].$  Among other things, we shall use the fact that for any smooth function 
$F$ vanishing at $0,$ the function 
 $F(a)$  can be bounded in $L^{8/3}(0,t;L^4)$ in terms of $\cX_4(t)$ (and thus of $\cX_p(t)$ if $p\leq4$). Indeed, leveraging the mean value formula and
 the fact that $\|a\|_{L^\infty}$ is small, we discover that
$$
\|F(a)\|_{L^4}\lesssim \|a\|_{L^4}$$
and thus, owing to \eqref{eq:interpo1}, the embedding $\dot B^{1/4}_{2,1}\hookrightarrow L^4$
and to $\|z^h\|_{L^4}\lesssim \|z\|^h_{\dot B^\sigma_{4,1}}$ for any $\sigma\geq0,$ 
\begin{equation}\label{eq:L8L4}
\|F(a)\|_{L^{8/3}_t(L^4)} \lesssim \|a\|_{L^{8/3}_t(L^4)} \lesssim 
 \|a\|_{L^{8/3}_t(\dot B^{1/4}_{2,1})}^\ell+  \|a\|_{L^{8/3}_t(\dot B^{1/4}_{4,1})}^h
\lesssim  \cX_4(t).
\end{equation}
Let us first bound  $\|(L(a)v_y)_y\|_{L^1_t(\dot B^{-1/2}_{2,1})}^\ell.$ To do so, we use Bony's decomposition:
$$ L(a)v_y= T_{L(a)} v_y +R(L(a), v_y)+ T_{v_y}L(a),$$
where $T$ and $R$ stand for the paraproduct and remainder operators defined  in \cite[Def. 2.45]{BCD}. 
In the one dimensional case, it is known that\footnote{The third result is stated in   \cite[Thm 2.52]{BCD}. 
The first two results are a slight modification of \cite[Thm 2.47]{BCD}:  we just have to change 
$\|\dot S_{j-1}u\|_{L^\infty}$ in $\|\dot S_{j-1} u\|_{L^4}$ in the proof therein.}
\begin{itemize} 
\item $T: L^4\times\dot B^\sigma_{4,1}\to \dot B^\sigma_{2,1}$ for any $\sigma\leq1/2$;
\item $T: \dot B^{\sigma'}_{4,1}\times \dot B^\sigma_{4,1}\to \dot B^{\sigma+\sigma'}_{2,1}$ for any $\sigma\in\R$ and $\sigma'\in\R_-$
with $\sigma+\sigma'\leq 1/2$;
\item $R:\dot B^{\sigma'}_{4,1}\times \dot B^\sigma_{4,1}\to \dot B^{\sigma+\sigma'}_{2,1}$ for any $(\sigma,\sigma')\in\R^2$ 
with $0<\sigma+\sigma'\leq1/2.$
\end{itemize}
Therefore, we have   
\begin{align*}\|T_{L(a)}v_y\|_{\dot B^{-1/2}_{2,1}}^\ell
 &\lesssim \|L(a)\|_{L^4} \|v_y\|_{\dot B^{-1/2}_{4,1}}
\lesssim \|a\|_{L^4}\bigl(\|v_y\|_{\dot B^{-1/4}_{2,1}}^\ell + \|v_y\|_{\dot B^{-1/2}_{4,1}}^h\bigr),\\
\|R(L(a),v_y)\|_{\dot B^{1/2}_{2,1}}^\ell
 &\lesssim \|L(a)\|_{\dot B^{1/4}_{4,1}} \|v_y\|_{\dot B^{1/4}_{4,1}}
\lesssim \|a\|_{\dot B^{1/4}_{4,1}} \|v_y\|_{\dot B^{1/4}_{4,1}},\\
\|T_{v_y}L(a)\|_{\dot B^{-1/2}_{2,1}}^\ell
 &\lesssim\|v_y\|_{\dot B^{-3/4}_{4,1}} \|L(a)\|_{\dot B^{1/4}_{4,1}} 
\lesssim \bigl(\|v_y\|_{\dot B^{-1/2}_{2,1}}^\ell + \|v_y\|_{\dot B^{-3/4}_{4,1}}^h\bigr)\|a\|_{\dot B^{1/4}_{4,1}}.\end{align*}
Hence, using also H\"older inequality with respect to time, and the embedding
$$
\dot B^\sigma_{2,1}\hookrightarrow \dot B^{\sigma-1/2+1/p}_{p,1}\hookrightarrow \dot B^{\sigma-1/4}_{4,1},\quad
2\leq p\leq 4,\quad \sigma\in\R,$$
we obtain
\begin{align*}\|T_{L(a)}v_y\|_{L^1_t(\dot B^{-1/2}_{2,1})}^\ell
&\lesssim \|a\|_{L^{8/3}_t(L^4)}\bigl(\|v_y\|_{L^{8/5}_t(\dot B^{-1/4}_{2,1})}^\ell + \|v_y\|_{L^{8/5}_t(\dot B^{-1/2}_{4,1})}^h\bigr),\\
\|R(L(a),v_y)\|_{L^1_t(\dot B^{1/2}_{2,1})}^\ell
&\lesssim \|a\|_{L^\infty_t(\dot B^{1/4}_{4,1})} \|v_y\|_{L^1_t(\dot B^{1/4}_{4,1})},\\\|T_{v_y}L(a)\|_{L^1_t(\dot B^{-1/2}_{2,1})}^\ell 
&\lesssim \bigl(\|v_y\|_{L^2_t(\dot B^{-1/2}_{2,1})}^\ell + \|v_y\|_{L^2_t(\dot B^{-3/4}_{4,1})}^h\bigr)\|a\|_{L^2_t(\dot B^{1/4}_{4,1})}.
\end{align*}
Remembering \eqref{eq:interpo1} and \eqref{eq:lfhf}, we conclude that 
\begin{equation}
\|(L(a)v_y)_y\|_{L^1_t(\dot B^{-1/2}_{2,1})}^\ell\lesssim \cX_4^2(t).
\end{equation}
For bounding  $\|(aK(a))_y\|_{L^1_t(\dot B^{-1/2}_{2,1})}^\ell,$ we again  resort to Bony's decomposition: 
$$a K(a)=T_a K(a)+R(a,K(a)) +T_{K(a)} a.$$ 
The remainder term is easy to handle: according to  \cite[Thm 2.52]{BCD}, we have
$$\begin{aligned}\| (R(a,K(a)))_y\|_{L^1_t(\dot B^{-1/2}_{2,1})}^\ell&\lesssim \|R(a,K(a))\|_{L^1_t(\dot B^{1/2}_{2,1})}\\
&\lesssim \|a\|_{L^2_t(\dot B^{1/4}_{4,1})}\|K(a)\|_{L^2_t(\dot B^{1/4}_{4,1})}\\&\lesssim  \|a\|_{L^2_t(\dot B^{1/4}_{4,1})}^2.
\end{aligned}$$
From Inequalities  \eqref{eq:lfhf},  \eqref{eq:embed} and the aforementioned results of continuity for $T,$ we get 
$$\|T_{K(a)} a^\ell\|_{\dot B^{1/2}_{2,1}}^\ell\lesssim  \|K(a)\|_{L^4} \|a^\ell\|_{\dot B^{1/2}_{4,1}}
\lesssim \|a\|_{L^4} \|a\|_{\dot B^{3/4}_{2,1}}^\ell,$$
$$\|T_{K(a)} a^h\|_{\dot B^{1/2}_{2,1}}^\ell\lesssim\|T_{K(a)} a^h\|_{\dot B^{1/4}_{2,1}}\lesssim \|K(a)\|_{L^4} \|a^h\|_{\dot B^{1/4}_{4,1}}
\lesssim \|a\|_{L^4} \|a\|^h_{\dot B^{1/4}_{4,1}}.$$
Therefore, we have
$$\begin{aligned}\|T_{K(a)} a^\ell\|_{L_t^1(\dot B^{1/2}_{2,1})}^\ell&\lesssim
 \|a\|_{L^{8/3}_t(L^4)} \|a\|_{L^{8/5}_t(\dot B^{3/4}_{2,1})}^\ell\lesssim \cX_4^2(t),\\
\|T_{K(a)} a^h\|_{L_t^1(\dot B^{1/2}_{2,1})}^\ell
&\lesssim  \|a\|_{L^{8/3}_t(L^4)} \|a\|_{L^{8/5}_t(\dot B^{1/4}_{4,1})}^h\lesssim \cX_4^2(t).\end{aligned}$$
Finally, to estimate the term $T_a K(a),$ we use the fact that 
$$K(a)=K(a^\ell) + K^h(a) \with K^h(a):= a^h\int_0^1 K'(a^\ell +\tau a^h)\,d\tau.$$
As $\dot B^{1/2}_{2,1}$ and $\dot B^{1/4}_{4,1}$ are stable by product and  left composition, and $\|a\|_{L^\infty}$ is small, we have 
$$\|K(a^\ell)\|_{\dot B^{1/2}_{4,1}}\lesssim \|a^\ell\|_{\dot B^{1/2}_{4,1}}\andf 
\|K^h(a)\|_{\dot B^{1/4}_{4,1}}\lesssim \|a\|_{\dot B^{1/4}_{4,1}}^h.$$ 
Hence the terms  $T_a K(a^\ell)$ and $T_a K^h(a)$ may be bounded like $T_{K(a)} a^\ell$ and 
$T_{K(a)} a^h,$ respectively. 
\smallbreak
In the end, reverting to \eqref{eq:NL1}, we get 
\begin{equation}\label{eq:quadratic}\cX_p(t)\leq C\bigl(\cX_{p,0}+ \cX_p^2(t)\bigr),\end{equation}
which, if $\cX_{p,0}$ is small enough, leads to  
\begin{equation}\label{eq:global}\cX_p(t)\leq 2C\cX_{p,0}.\end{equation}
Using the scaling properties of Besov spaces pointed out in \eqref{eq:rescaling},
one can now get 
 Inequality~\eqref{eq:final} in the general case.

\subsection{Stability with respect to the data, and uniqueness} 

In this subsection, we prove the following result which, obviously, implies the uniqueness part of Theorems  \ref{thm:main0} and \ref{thm:main1}.
\begin{prop} \label{p:uniqueness}
 Let $(a^1,v^1)$ and $(a^2,v^2)$ be two solutions of  Equations \eqref{eq:NSClaglin} on $[0,T]\times \R.$
 Assume that there exists  $p\in[1,\infty)$ such that  for $i=1,2,$  we have
 $$
 a^i\in\cC([0,T];\dot B^{1/p}_{p,1})\andf  v^i\in \cC([0,T];\dot B^{-1+1/p}_{p,1})\cap L^1(0,T;\dot B^{1+1/p}_{p,1}).$$
 There exist two  constants $c$ and $C$ depending only on $p$ and on the functions $K$ and $L,$ and a constant $C_{T}$ 
 depending also on $\|(a^1,a^2)\|_{L^\infty(0,T\times\R)}$  such that if
 \begin{equation}\label{eq:smalluniq}
 \sup_{t\in[0,T]} \|a^2(t)\|_{\dot B^{1/p}_{p,1}}\leq c,\end{equation}
then the functions $\da:=a^2-a^1$ and $\dv:=v^2-v^1$ satisfy  for all $t\in[0,T]$:
\begin{multline}\label{eq:uniqueness}
\|\da(t)\|_{\dot B^{1/p}_{p,1}}+\|\dv(t)\|_{\dot B^{-1+1/p}_{p,1}}+\int_0^t\|\dv\|_{\dot B^{1+1/p}_{p,1}}
\\\leq C\bigl(\|\da(0)\|_{\dot B^{1/p}_{p,1}}+\|\dv(0)\|_{\dot B^{-1+1/p}_{p,1}}\bigr)\exp\biggl(C_{T}\int_0^t\Bigl(\|v^1_y\|_{\dot B^{1/p}_{p,1}}
+1+\|a^1\|_{\dot B^{1/p}_{p,1}}^2\Bigr)\biggr)\cdotp\end{multline}
\end{prop}
\begin{proof}
Let us set $\dK:=K(a^2)-K(a^1)$ and  $\dL:=L(a^2)-L(a^1).$ Since both solutions satisfy Equations \eqref{eq:NSClaglin}, the 
pair $(\da,\dv)$ solves:
 \begin{equation}\label{eq:NSCuniq}\left\{
 \begin{array}{l}
  \da_t = \dv_y,\\[5pt]
\dv_t -\dv_{yy}=\da_y+\dh_y
 \end{array}\right.\end{equation}
with $\dh:=a^1\dK+\da K(a^2)+\dL\,v_y^1+L(a^2)\dv_y.$
\smallbreak
It is completely obvious that for all $t\in[0,T],$ 
\begin{equation}\label{eq:da}\|\da(t)\|_{\dot B^{1/p}_{p,1}}\leq \|\da(0)\|_{\dot B^{1/p}_{p,1}}+\int_0^t\|\dv_y\|_{\dot B^{1/p}_{p,1}}
\end{equation}
and we have using parabolic maximal regularity,
\begin{equation}\label{eq:dv}\|\dv(t)\|_{\dot B^{-1+1/p}_{p,1}}+\int_0^t\|\dv\|_{\dot B^{1+1/p}_{p,1}}\lesssim 
\|\dv(0)\|_{\dot B^{-1+1/p}_{p,1}}+\int_0^t\|\dh\|_{\dot B^{1/p}_{p,1}}.\end{equation}
Now, basic product and composition estimates (see \cite[Chap. 2]{BCD}) guarantee us that 
$$\begin{aligned} \|\dh\|_{\dot B^{1/p}_{p,1}}&\leq C\Bigl(
\|a^1\|_{\dot B^{1/p}_{p,1}} \|\dK\|_{\dot B^{1/p}_{p,1}} \!+\!\|\da\|_{\dot B^{1/p}_{p,1}} \| K(a^2)\|_{\dot B^{1/p}_{p,1}}\! +\!
\|\dL\|_{\dot B^{1/p}_{p,1}} \|v_y^1\|_{\dot B^{1/p}_{p,1}} \\&\hspace{8cm}
+\!\|L(a^2)\|_{\dot B^{1/p}_{p,1}} \|\dv_y\|_{\dot B^{1/p}_{p,1}}\Bigr) \\
&\leq C_{T}\Bigl(\bigl(1\!+\!\|(a^1,a^2)\|_{\dot B^{1/p}_{p,1}}\bigr)
\|\da\|_{\dot B^{1/p}_{p,1}}\|(a^1,a^2,v^1_y)\|_{\dot B^{1/p}_{p,1}}+\|a^2\|_{\dot B^{1/p}_{p,1}}
\|\dv\|_{\dot B^{1+1/p}_{p,1}}\Bigr)\cdotp \end{aligned}$$
Combining Inequalities \eqref{eq:da} and \eqref{eq:dv} and assuming that $c$ in \eqref{eq:smalluniq} is small enough gives
\begin{multline}\label{eq:dadv}\|\da(t)\|_{\dot B^{1/p}_{p,1}}+\|\dv(t)\|_{\dot B^{-1+1/p}_{p,1}}+\int_0^t\|\dv\|_{\dot B^{1+1/p}_{p,1}}\leq C_{T}\Bigl(
\|\da(0)\|_{\dot B^{1/p}_{p,1}}+\|\dv(0)\|_{\dot B^{-1+1/p}_{p,1}}\\
+\int_0^t\bigl(1+\|a^1\|_{\dot B^{1/p}_{p,1}}^2+\|v^1_y\|_{\dot B^{1/p}_{p,1}}\bigr)\|\da\|_{\dot B^{1/p}_{p,1}}\Bigr)\end{multline}
and Inequality \eqref{eq:uniqueness} then stems from Gronwall lemma.
\end{proof}

\subsection{The proof of Theorem \ref{thm:main1}} 
The uniqueness being ensured by the above proposition, we give some hint on the proof of existence. 
It is achieved by means of  the classical  scheme consisting in solving  inductively a sequence of
 linear equations (corresponding to  \eqref{eq:NSClaglin}), proving all-time uniform estimates in the  desired solution space, then checking that the constructed sequence is a Cauchy one in a slightly larger functional space which,  
 nevertheless, contains enough regularity  to pass to the limit in the approximate equations.  
 
 \subsubsection*{Step 1. Construction of a sequence of approximate solutions}
 
 We define the first term $(a^0,v^0)$ of the sequence to be the solution of \eqref{eq:NSClaglin} with $g\equiv0$
 and initial data $(a_0,v_0).$ Then, once $(a^n,v^n)$ has been constructed we take $(a^{n+1},v^{n+1})$ to be the solution of
  \begin{equation}\label{eq:NSClaglinn}\left\{
 \begin{array}{l}
  a_t^{n+1} - v_y^{n+1} =0,\\[5pt]
v_t^{n+1} -  a_y^{n+1} -v_{yy}^{n+1}= (a^n K(a^n))_y+(L(a^n) v_y^n)_y
 \end{array}\right.\end{equation}
 with $(a^{n+1},v^{n+1})|_{t=0}=(a_0,v_0).$ 
\medbreak
 Note that all the terms of the sequence can be computed from the previous one by means of the variation of constant formula
 (everything can be made explicit on the Fourier space, see the Appendix). 
 In particular, one can show inductively that  the terms of the sequence are globally defined and belong to the desired space. 
 
 \subsubsection*{Step 2. Uniform estimates} 
 Starting from \eqref{eq:NL1}, one can reproduce faithfully the estimates of $g$ leading to \eqref{eq:quadratic}, getting eventually
 (with obvious notation):
 $$ \forall t\in\R_+,\; \cX_p^{n+1}(t)\leq C\Bigl(\cX_{p,0}+\bigl(\cX_p^n(t)\bigr)^2\Bigr)\cdotp$$
 This implies that if $\cX_{p,0}$ is small enough, then we have 
 \begin{equation}\label{eq:unifbound}  \forall t\in\R_+,\; \cX_p^{n}(t)\leq 2C\cX_{p,0}.\end{equation}
 
 \subsubsection*{Step 3. Convergence of the sequence}
We claim that  $(a^n,v^n)_{n\in\N}$ converges in the space
$$\cC(\R_+;\dot B^{1/p}_{p,1})\times\bigl(\cC(\R_+;\dot B^{-1+1/p}_{p,1})\cap L^1_{loc}(\R_+;\dot B^{1+1/p}_{p,1})\bigr)\cdotp$$
Indeed, it suffices to show that   $(a^n,v^n)_{n\in\N}$  is a Cauchy sequence in the space
\begin{equation}\label{eq:FT}
F_T:=\cC([0,T];\dot B^{1/p}_{p,1})\times\bigl(\cC([0,T];\dot B^{-1+1/p}_{p,1})\cap L^1(0,T;\dot B^{1+1/p}_{p,1})\bigr)\quad
\hbox{for all }\ T>0.\end{equation}
To do so, we just have to modify slightly  the proof of Proposition \ref{p:uniqueness}: let
 $\da^n:=a^{n+1}-a^n$ and $\dv^n:=v^{n+1}-v^n.$ Then, 
we observe that for all $n\in\N,$ we have  $\da^n|_{t=0}=\dv^n|_{t=0}=0$ and 
  \begin{equation*}\left\{
 \begin{array}{l} \da_t^{n} - \dv_y^{n} =0,\\[5pt]
\dv_t^{n} -  \da_y^{n} -\dv_{yy}^{n}= (a^n K(a^n)-a^{n-1}K(a^{n-1}))_y+\bigl(L(a^n) v_y^n-L(a^{n-1})v_y^{n-1}\bigr)_y.
 \end{array}\right.\end{equation*}
 Hence, remembering that, thanks to \eqref{eq:unifbound}, all terms $a^n$ are small in $L^\infty(\R_+;\dot B^{1/p}_{p,1})$
 and arguing exactly as for getting \eqref{eq:dadv}, we discover that
 \begin{multline*}\|\da^n(t)\|_{\dot B^{1/p}_{p,1}}+\|\dv^n(t)\|_{\dot B^{-1+1/p}_{p,1}}+\int_0^t\|\dv^n\|_{\dot B^{1+1/p}_{p,1}}\leq 
 C\biggl(\int_0^t\|\da^n\|_{\dot B^{1/p}_{p,1}}\\+
   \int_0^t\bigl(\|a^{n-1}\|_{\dot B^{1/p}_{p,1}}+\|a^n\|_{\dot B^{1/p}_{p,1}} +\|v^{n-1}_y\|_{\dot B^{1/p}_{p,1}}\bigr)\|\da^{n-1}\|_{\dot B^{1/p}_{p,1}}+\int_0^t\|a^{n}\|_{\dot B^{1/p}_{p,1}}\|\dv_y^{n-1}\|_{\dot B^{1/p}_{p,1}}\biggr)\cdotp\end{multline*}
Summing up on $n\in\N,$ applying Gronwall lemma then using \eqref{eq:unifbound}, one can conclude that
 $$\sum_{n\in\N} \underset{t\in[0,T]}\sup\Bigl(\|\da^n(t)\|_{\dot B^{1/p}_{p,1}}+\|\dv^n(t)\|_{\dot B^{-1+1/p}_{p,1}}\Bigr)+
 \sum_{n\in\N}\int_0^T\|\dv^n\|_{\dot B^{1+1/p}_{p,1}}<\infty,\qquad T>0.$$
 This implies that $(a^n,v^n)_{n\in\N}$ is a Cauchy sequence in the space $F_T,$ 
 and thus converges to some limit $(a,v)\in F_T,$ for all $T>0.$
 
  \subsubsection*{Step 4. Passing to the limit in the equations}
 The strong convergence obtained in the previous step allows to pass to the limit in \eqref{eq:NSClaglinn}, and 
 $(a,v)$ is thus a solution of \eqref{eq:NSClag}, with initial data $(a_0,v_0).$ 
 The uniform bound \eqref{eq:unifbound} ensures in addition that $(a,v)$ belongs to the 
  space described in Theorem \ref{thm:main1}, up to the fact that classical functional analysis does not
  ensure time continuity and the $L^1$-in-time integrability. 
  To recover these latter properties, one can argue like in e.g. \cite[Chap. 10]{BCD}. 
  This completes the proof. \qed


\section{Time decay estimates} \label{s:decay}
 
The goal is to prove Theorem \ref{thm:main2}.
Under the hypothesis \eqref{eq:reduction}, we consider a global solution $(\eta=1+a,v)$ given by Theorem \ref{thm:main1}. We want to establish 
 that if Condition \eqref{eq:smalldata} holds true (with, possibly, a smaller constant $c$), 
then $(a,v)$ satisfies the decay inequality \eqref{eq:decay}. 

 \subsection{Decay of the low frequencies}
 
To estimate the term $\cD^\ell$ defined in  \eqref{eq:decay}, 
 we start from \eqref{eq:lffinal} which allows to get 
 $$\frac 12\frac d{dt}\bigl(\langle t\rangle\cL_j\bigr)^2+c_B2^{2j} \bigl(\langle t\rangle\cL_j\bigr)^2\leq C_B\bigl(\langle t\rangle\cL_j\bigr)
 (\langle t\rangle\|g_j\|_{L^2}) +t\cL_j^2$$
and thus after multiplying both sides by $2^{3j},$
 $$\frac 12\frac d{dt}\bigl(\langle t\rangle2^{\frac{3j}2}\cL_j\bigr)^2
 +c_B \bigl(\langle t\rangle2^{\frac{5j}2}\cL_j\bigr)^2\leq C_B\bigl(\langle t\rangle2^{\frac{5j}2}\cL_j\bigr)
 (\langle t\rangle2^{\frac j2}\|g_j\|_{L^2}) +(t2^{\frac{5j}2}\cL_j)(2^{\frac{j}2}\cL_j).$$
Then,  we use Young inequality to absorb $\langle t\rangle2^{\frac{5j}2}\cL_j$ and $t2^{\frac{5j}2}\cL_j$
by the left-hand side, and integrate with respect to time, getting for all $t\in\R_+,$ 
$$\sup_{\tau\in[0,t]}\bigl(\langle \tau\rangle2^{\frac{3j}2}\cL_j(\tau)\bigr)^2+c_B\int_0^t \bigl(\langle \tau\rangle2^{\frac{5j}2}\cL_j\bigr)^2
\lesssim \bigl(2^{\frac{3j}2}\cL_j(0)\bigr)^2
+\int_0^t \bigl(\langle \tau\rangle2^{\frac{j}2}\|g_j\|_{L^2})^2
+\int_0^t \bigl(2^{\frac{j}2}\cL_j\bigr)^2\cdotp$$
Summing up on $j\leq j_0$ and remembering \eqref{eq:equiv} and the definition of $\cD^\ell,$ we thus get
$$\cD^\ell(t)\lesssim \|(a_0,v_0)\|^\ell_{\dot B^{3/2}_{2,1}}
 +\|\langle \tau\rangle g\|_{\wt L_t^2(\dot B^{1/2}_{2,1})}^\ell +\|(a,v)\|_{\wt L_t^2(\dot B^{1/2}_{2,1})}^\ell.$$
Bounding the last term according to \eqref{eq:interpo1} and \eqref{eq:global}, we conclude that
 \begin{equation}\label{eq:decaylf}
\cD^\ell(t)\lesssim \cX_{2,0} 
  +\|\langle \tau\rangle g\|_{\wt L_t^2(\dot B^{1/2}_{2,1})}^\ell\quad\hbox{for all }\ t\in\R_+.\end{equation}
 To bound $g$, we shall use the decomposition:
 \begin{equation}\label{eq:decompo}
  a^2=(a^\ell)^2+2a^\ell a^h+(a^h)^2\andf av_y= a^\ell v_y^\ell+ a^h v_y^\ell +a v_y^h.\end{equation}
 In a first time, we assume that $K(a)=L(a)=a,$ so that  $g=(a^2)_y+(av_y)_y.$
 Then,   using  \eqref{eq:lfhf}, leveraging Bony's decomposition 
 and results of continuity for the paraproduct and remainder operators
 (use \cite[Thm 2.47 and 2.52]{BCD})  yields 
  $$\begin{aligned}\|\langle \tau\rangle\partial_y((a^\ell)^2)\|_{\wt L_t^2(\dot B^{1/2}_{2,1})}^\ell
   &\lesssim  \|\langle \tau\rangle(a^\ell)^2\|_{\wt L_t^2(\dot B^{3/2}_{2,1})}
 \lesssim  \|a\|^\ell_{\wt L_t^\infty(\dot B^{-1/2}_{2,1})}\|\langle \tau\rangle a\|_{\wt L_t^2(\dot B^{5/2}_{2,1})}^\ell,\\
 \|\langle \tau\rangle\partial_y(a^\ell a^h)\|_{\wt L_t^2(\dot B^{1/2}_{2,1})}^\ell   &\lesssim 
  \|\langle \tau\rangle a^\ell a^h\|_{\wt L_t^2(\dot B^{1/2}_{2,1})}\lesssim  \|\langle \tau\rangle a\|^h_{\wt L_t^\infty(\dot B^{1/2}_{2,1})}
  \|a\|_{\wt L_t^2(\dot B^{1/2}_{2,1})}^\ell,\\
 \|\langle \tau\rangle\partial_y((a^h)^2)\|_{\wt L_t^2(\dot B^{1/2}_{2,1})}^\ell
   &\lesssim  \|\langle \tau\rangle(a^h)^2\|_{\wt L_t^2(\dot B^{1/2}_{2,1})}\lesssim  \|\langle \tau\rangle a\|^h_{\wt L_t^\infty(\dot B^{1/2}_{2,1})}
  \|a\|_{\wt L_t^2(\dot B^{1/2}_{2,1})}^h.
   \end{aligned}$$
Hence
 \begin{equation}\label{eq:decaylf1} \|\langle \tau\rangle \partial_y(a^2)\|_{\wt L_t^2(\dot B^{1/2}_{2,1})}^\ell
 \lesssim \cX_{2}(t)\bigl(\cD^\ell(t)+\cD^h_a(t)\bigr)\cdotp
\end{equation}
In the same spirit, we have 
 $$\begin{aligned}\|\langle \tau\rangle \partial_y(  a^\ell v_y^\ell  )\|_{\wt L^2_t(\dot B^{1/2}_{2,1})}^\ell
 &\lesssim  \|\langle \tau\rangle (  a^\ell v_y^\ell  )\|_{\wt L^2_t(\dot B^{1/2}_{2,1})}^\ell\\
&\lesssim \|a\|^\ell_{\wt L_t^\infty(\dot B^{-1/2}_{2,1})}\|\langle\tau\rangle v_y\|^\ell_{\wt L_t^2(\dot B^{3/2}_{2,1})}
 + \|\langle \tau\rangle a\|^\ell_{\wt L_t^\infty(\dot B^{3/2}_{2,1})}\|v_y\|^\ell_{\wt L_t^2(\dot B^{-1/2}_{2,1})},\\
 \|\langle \tau\rangle \partial_y(   a^h v_y^\ell  )\|_{\wt L^2_t(\dot B^{1/2}_{2,1})}^\ell
 &\lesssim \|\langle\tau\rangle a^h v_y^\ell\|_{\wt L^2_t(\dot B^{1/2}_{2,1})}^\ell\lesssim \|a\|^h_{\wt L^2_t(\dot B^{1/2}_{2,1})}
 \|\langle \tau\rangle v_y\|_{\wt L_t^\infty(\dot B^{1/2}_{2,1})}^\ell,\\
\|\tau\partial_y( av_y^h )\|_{\wt L^2_t(\dot B^{1/2}_{2,1})}^\ell
 &\lesssim   \|\tau ( av_y^h )\|_{\wt L^2_t(\dot B^{1/2}_{2,1})}^\ell\lesssim \|a\|_{\wt L^2_t(\dot B^{1/2}_{2,1})}
 \|\tau v\|_{\wt L_t^\infty(\dot B^{3/2}_{2,1})}^h.\end{aligned}$$
 Combining the product law $\dot B^{1/2}_{2,1}\times \dot B^{-1/2}_{2,1}\to \dot B^{-1/2}_{2,\infty}$ and
the fact that for  $s<s',$ we have
 \begin{equation}
 \|z\|^\ell_{\dot B^{s'}_{2,1}}\lesssim  \|z\|^\ell_{\dot B^{s}_{2,\infty}},\end{equation}
 we may write 
 $$
 \|\d_y(av_y^h)\|_{\wt L_t^2(\dot B^{1/2}_{2,1})}^\ell \lesssim 
  \|av_y^h\|_{\wt L_t^2(\dot B^{-1/2}_{2,\infty})}\lesssim  \|a\|_{\wt L_t^\infty(\dot B^{1/2}_{2,1})}
   \|v_y\|_{\wt L_t^2(\dot B^{-1/2}_{2,1})}. $$
 Hence
\begin{equation}\label{eq:decaylf2}\|\langle \tau \rangle (a v_y)_y\|_{\wt L_t^2(\dot B^{1/2}_{2,1})}^\ell
\lesssim \cX_2(t) \bigl(\cD^\ell(t)+\cD^h_a(t)+\wt\cD^h_v(t)+\cX_2(t)\bigr)\cdotp\end{equation}
Reverting to \eqref{eq:decaylf} and remembering \eqref{eq:global} and that $\cX_{2,0}$ is small, we end up with 
\begin{equation}\label{eq:Dell}\cD^\ell(t)\lesssim \cX_{2,0} (1+ \cD^\ell(t)+\cD_a^h(t)+\wt\cD^h_v(t)\bigr)\cdotp\end{equation} 
For  general viscosity and pressure functions, we have to handle the additional 
terms $(a^2\wt K(a))_y$ and $(a\wt L(a)v_y)_y$ with $\wt K(0)=\wt L(0)=0.$
To bound $(a^2\wt K(a))_y,$ it suffices to write that
\begin{align*}
\|\langle \tau\rangle\partial_y(a^2\wt K(a))\|_{\wt L_t^2(\dot B^{1/2}_{2,1})}^\ell
&\lesssim \|\langle \tau\rangle a^2\wt K(a)\|_{\wt L_t^2(\dot B^{1/2}_{2,1})}\\
&\lesssim  \|\langle \tau\rangle^{1/2} a\|_{\wt L_t^\infty(\dot B^{1/2}_{2,1})}^2
 \|a\|_{\wt L_t^2(\dot B^{1/2}_{2,1})}.
 \end{align*}
Since
\begin{align}\label{eq:keyinterpo}
\|\langle t\rangle^{1/2}a\|_{\wt L_t^\infty(\dot B^{1/2}_{2,1})}&\leq \|\langle t\rangle^{1/2}a\|_{\wt L_t^\infty(\dot B^{1/2}_{2,1})}^\ell
+\|\langle t\rangle^{1/2}a\|_{\wt L_t^\infty(\dot B^{1/2}_{2,1})}^h\nonumber\\
&\lesssim \sqrt{ \|a\|_{\wt L_t^\infty(\dot B^{-1/2}_{2,1})}^\ell\|\langle t\rangle a\|_{\wt L_t^\infty(\dot B^{3/2}_{2,1})}^\ell}
+ \sqrt{ \|a\|^h_{\wt L_t^\infty(\dot B^{1/2}_{2,1})}\|\langle t\rangle a\|_{\wt L_t^\infty(\dot B^{1/2}_{2,1})}^h}\nonumber\\
&\lesssim \sqrt{\cX_2(t)(\cD^\ell(t) \!+\! \cD^h_a(t))}, \end{align}
we get
$$\|\langle \tau\rangle\partial_y(a^2\wt K(a))\|_{\wt L_t^2(\dot B^{1/2}_{2,1})}^\ell
\lesssim  (\cX_2(t))^2\bigl(\cD^\ell(t)+\cD_a^h(t)\bigr)\cdotp$$
To bound  $(a\wt L(a)v_y)_y,$ 
the general  principle is to  first remove  the space derivative, then
to decompose  $av_y$ according to \eqref{eq:decompo}, and to argue as before.
For example:
 $$\begin{aligned}\|\langle \tau\rangle\partial_y(a^\ell v_y^\ell\wt L(a))\|_{\wt L_t^2(\dot B^{1/2}_{2,1})}^\ell
 &\lesssim \|\langle \tau\rangle a^\ell v_y^\ell\wt L(a)\|_{\wt L_t^2(\dot B^{1/2}_{2,1})}\\
  &\lesssim    \|\langle \tau\rangle (  a^\ell v_y^\ell  )\|_{\wt L^2_t(\dot B^{1/2}_{2,1})}
   \|a\|_{\wt L^\infty_t(\dot B^{1/2}_{2,1})}   \\
   \lesssim  \bigl(\|a\|&^\ell_{\wt L_t^\infty(\dot B^{-1/2}_{2,1})}\|\langle\tau\rangle v_y\|^\ell_{\wt L_t^2(\dot B^{3/2}_{2,1})}
\!+\! \|\langle \tau\rangle a\|^\ell_{\wt L_t^\infty(\dot B^{3/2}_{2,1})}\|v_y\|^\ell_{\wt L_t^2(\dot B^{-1/2}_{2,1})}\bigr)
   \|a\|_{\wt L^\infty_t(\dot B^{1/2}_{2,1})}.\end{aligned}$$
 The  terms corresponding to $a^hv_y^\ell$ and $a v_y^h$ can  treated in the same way. 
Since $\|a\|_{\wt L^\infty_t(\dot B^{1/2}_{2,1})}$ is small, one still gets \eqref{eq:decaylf1},  \eqref{eq:decaylf2} and thus  \eqref{eq:Dell}.

 \subsection{Decay of the high frequencies}

In order to bound $\cD^h_a$ and $\cD^h_v,$ we start from the observation that for all $\sigma\in[1,3/2],$ the pair
$(a,w_y)$ (with $w$ defined in \eqref{eq:hf1}) satisfies 
\begin{equation}\label{eq:decayS}\left\{\begin{array}{l}
(t^{\sigma}a)_t+t^{\sigma}a=\sigma t^{\sigma-1}a +t^{\sigma} w_y,\\[1ex]
(t^{\sigma}w_y)_t-(t^{\sigma}w_y)_{yy} =t^{\sigma}w_y+\sigma t^{\sigma-1}w_y +t^{\sigma}g_y -t^{\sigma}a.\end{array}\right.\end{equation}
From the equation of $a,$ it is obvious that for all $t\in\R_+,$
$$\|\tau^{\sigma}a\|_{\wt L_t^\infty(\dot B^{1/2}_{2,1})}^h
\leq \sigma\|\tau^{\sigma-1}a\|_{\wt L_t^\infty(\dot B^{1/2}_{2,1})}^h
+\|\tau^{\sigma}w_y\|_{\wt L_t^\infty(\dot B^{1/2}_{2,1})}^h.$$
For $t\geq2,$ the first term of the right-hand side may be absorbed by the left-hand side, which results in
$$\|\tau^{\sigma}a\|_{\wt L_t^\infty(\dot B^{1/2}_{2,1})}^h
\leq 4\|\tau^{\sigma}w_y\|_{\wt L_t^\infty(\dot B^{1/2}_{2,1})}^h.$$
Since we already know from Theorem \ref{thm:main1} that
$$\|a\|_{\wt L_2^\infty(\dot B^{1/2}_{2,1})}^h\lesssim \cX_{2,0},$$
one can conclude that, for all $t\in\R_+,$
\begin{equation}\label{eq:decayahf}
\|\langle\tau\rangle^{\sigma}a\|_{\wt L_t^\infty(\dot B^{1/2}_{2,1})}^h
\lesssim \cX_{2,0} + \|\tau^{\sigma}w_y\|_{\wt L_t^\infty(\dot B^{1/2}_{2,1})}^h.\end{equation}
From the second equation of \eqref{eq:decayS} and maximal regularity estimates, we get
\begin{multline*}
\|\tau^{\sigma} w_y\|_{\wt L_t^\infty(\dot B^{1/2}_{2,1})}^h\leq
\|\tau^{\sigma}w_y\|_{\wt L_t^\infty(\dot B^{-3/2}_{2,1})}^h+\sigma \|\tau^{\sigma-1}w_y\|_{\wt L_t^\infty(\dot B^{-3/2}_{2,1})}^h
 \\+\|\tau^{\sigma}g_y\|_{\wt L_t^\infty(\dot B^{-3/2}_{2,1})}^h +\|\tau^{\sigma}a\|_{\wt L_t^\infty(\dot B^{-3/2}_{2,1})}^h.\end{multline*}
For $t\geq2,$ owing to the high frequencies cut-off, the first two terms of the right-hand side may be absorbed by the left-hand side. 
For $t\leq2,$ they may be just bounded by $\cX_{2,0},$ according to Theorem \ref{thm:main1}. 
Hence, putting together with \eqref{eq:decayahf}, we end up with 
\begin{equation}\label{eq:decayhf}
\cD_a^h(t)+\wt\cD_w^h(t)\lesssim \cX_{2,0}+\|\tau^{\sigma}g\|_{\wt L_t^\infty(\dot B^{-1/2}_{2,1})}^h.\end{equation}
If $K(a)=L(a)=a,$ then the nonlinear term $g$ can be bounded from product laws, \eqref{eq:lfhf} in terms of
$\cX_2,$ $\cD^\ell,$ $\cD^h_a$ and $\wt\cD^h_v$ as follows:
$$\begin{aligned}\|t^{\sigma}\partial_y((a^\ell)^2)\|_{\wt L_t^\infty(\dot B^{-1/2}_{2,1})}^h&\lesssim 
\|t^{\sigma}(a^\ell)^2\|_{\wt L_t^\infty(\dot B^{1/2}_{2,1})}^h\\
&\lesssim \|t^{\sigma}(a^\ell)^2\|_{\wt L_t^\infty(\dot B^{3/2}_{2,1})}\\
&\lesssim \|t a\|_{\wt L_t^\infty(\dot B^{3/2}_{2,1})}^\ell \|t^{\sigma-1} a\|^\ell_{\wt L_t^\infty(\dot B^{1/2}_{2,1})},\end{aligned}
$$
$$\begin{aligned}\|t^{\sigma}(a a^h)_y\|_{\wt L_t^\infty(\dot B^{-1/2}_{2,1})}^h&\lesssim 
\|t^{\sigma} a a^h\|_{\wt L_t^\infty(\dot B^{1/2}_{2,1})}^h\\
&\lesssim \|a \|_{\wt L_t^\infty(\dot B^{1/2}_{2,1})}
\|t^{\sigma} a\|_{\wt L_t^\infty(\dot B^{1/2}_{2,1})}^h,\end{aligned}$$
$$\begin{aligned}
\|t^{\sigma}(a v_y^\ell)_y\|_{\wt L_t^\infty(\dot B^{-1/2}_{2,1})}^h&\lesssim 
\|t^{\sigma} a v_y^\ell\|_{\wt L_t^\infty(\dot B^{1/2}_{2,1})}^h\\
&\lesssim  \|t^{\sigma-1} a\|_{\wt L^\infty_t(\dot B^{1/2}_{2,1})}\|t v\|^\ell_{\wt L_t^\infty(\dot B^{3/2}_{2,1})},\end{aligned}$$
$$\begin{aligned}
\|t^{\sigma}\partial_y(a v_y^h)\|_{\wt L_t^\infty(\dot B^{-1/2}_{2,1})}^h&\lesssim 
\|t^{\sigma} a  v_y^h\|_{\wt L_t^\infty(\dot B^{1/2}_{2,1})}^h\\
&\lesssim \|a\|_{\wt L_t^\infty(\dot B^{1/2}_{2,1})} \|t^{\sigma} v_y\|_{\wt L_t^\infty(\dot B^{1/2}_{2,1})}^h.
\end{aligned}$$
Bounding  the terms $t^{\sigma-1}a$  by means of \eqref{eq:keyinterpo} (remember that $0\leq\sigma-1\leq1/2$), 
and reverting to \eqref{eq:decayhf} gives 
\begin{equation}\label{eq:decay3}\cD_a^h(t)+\wt\cD_w^h(t)\lesssim \cX_{2,0}+\sqrt{\cX_2(t)(\cD^\ell(t)\! + \!\cD^h_a(t))}\,\cD^\ell(t)
+\cX_2(t)\bigl(\cD_a^h(t)+\wt\cD_v^h(t)\bigr)\cdotp
 \end{equation}
For general pressure and viscosity functions, one has to handle in addition the terms
$(a^2\wt K(a))_y$ and $(a\wt L(a)v_y)_y$ with $\wt K(a)=\wt L(a)=0.$ 
They do not represent any difficulty as we have
$$\begin{aligned}
\|t^{\sigma}(a^2\wt K(a))_y\|_{\wt L^\infty_t(\dot B^{-1/2}_{2,1})}& \lesssim
\|t^{\sigma}a^2\wt K(a)\|_{\wt L^\infty_t(\dot B^{1/2}_{2,1})} \\
&\lesssim \|t^{\sigma}a^2\|_{\wt L^\infty_t(\dot B^{1/2}_{2,1})} \|a\|_{\wt L^\infty_t(\dot B^{1/2}_{2,1})}\\
&\lesssim   \|t^{\sigma}a^2\|_{\wt L^\infty_t(\dot B^{1/2}_{2,1})} \cX_2(t),\end{aligned}$$
$$\begin{aligned}
\|t^{\sigma}(a\wt L(a)v_y)_y\|_{\wt L^\infty_t(\dot B^{-1/2}_{2,1})}& \lesssim
\|t^{\sigma}(a v_y)\,\wt L(a)\|_{\wt L^\infty_t(\dot B^{1/2}_{2,1})} \\
&\lesssim \|t^{\sigma}av_y\|_{\wt L^\infty_t(\dot B^{1/2}_{2,1})} \|a\|_{\wt L^\infty_t(\dot B^{1/2}_{2,1})}\\
&\lesssim   \|t^{\sigma}a v_y\|_{\wt L^\infty_t(\dot B^{1/2}_{2,1})} \cX_2(t).\end{aligned}$$
The terms containing $a^2$ and $av_y$ may be bounded exactly as above. Since $\cX_2$ is small,  one still 
gets \eqref{eq:decay3}. 
To conclude the proof of Inequality \eqref{eq:decay}, it suffices to observe that owing to  $w_y-v_y=a,$  we have
$$\bigl\|\tau^{\sigma}(w_y-v_y)\|_{\wt L_t^\infty(\dot B^{1/2}_{2,1})} = \bigl\|\tau^{\sigma}a\bigr\|_{\wt L_t^\infty(\dot B^{1/2}_{2,1})}
\leq \cD_a^h(t).$$
Consequently, $\wt\cD_w^h$ can be replaced by $\wt\cD_v^h$ in \eqref{eq:decay3}. Then, putting together 
 with \eqref{eq:Dell} and using the fact that
$\cX_{2}$ remains small for all time, we get the desired statement. \qed


\section{The diffusive limit}\label{s:diffusive}

This section is dedicated to proving Theorem \ref{thm:main3}.

The first ingredient is the estimate provided by Theorem \ref{thm:main1}: we observe that
if $(\check \eta,\check v)$  is a solution of \eqref{eq:NL3}, then it also 
fulfills equations \eqref{eq:NL2} with parameters ${\rm Ma}\,\bar\nu^{-1}$ and $\bar\nu^2$
(instead of ${\rm Ma}$ and $\bar\nu$) and data $(\bar\eta+a_0,\bar\nu v_0).$  Hence, 
Condition \eqref{eq:smalldata} becomes 
\begin{equation}\label{eq:smalldata2}
 {\rm Ma}^{-1}\|a_0\|_{\dot B^{-1/2}_{2,1}} +\bar\nu \|a_0\|_{\dot B^{1/2}_{2,1}} 
 +  \|v_0\|_{\dot B^{-1/2}_{2,1}}\leq c\bar\eta\bar\nu,\end{equation}
 and Theorem \ref{thm:main1} guarantees that   $(\check \eta,\check v)$ is indeed global and satisfies Inequality \eqref{eq:finalbis}. 
\smallbreak
To pass to the limit in \eqref{eq:NL3}, it  looks that
we need a  better control on the low frequencies of the solutions. 
This motivates us to require that, in addition, $(a_0,v_0)$ belongs to 
$\dot B^{-\sigma}_{2,1}$ for some $\sigma\in(1/2,3/2).$ 
Note that, by interpolation, owing to $a_0\in\dot B^{1/2}_{2,1},$ this implies that $a_0\in\dot B^{1-\sigma}_{2,1}.$ 
\smallbreak
Let us introduce the following notation (where the dependency on ${\rm Ma}$ is omitted):
$$ \cI_{0,\bar\nu}^{s}:=\check\nu^{-1}\|a_0\|_{\dot B^{s}_{2,1}} +  \|a_0\|_{\dot B^{1+s}_{2,1}} 
 +  \bar\nu^{-1}\|v_0\|_{\dot B^{s}_{2,1}}.$$
We claim that provided  \eqref{eq:smalldata2} holds true, then  we have for all $t\in\R_+,$   
  \begin{multline}\label{eq:sigma}
 \check\nu^{-1}\|\check a(t)\|_{\dot B^{-\sigma}_{2,1}}+\|\check a(t)\|_{\dot B^{1-\sigma}_{2,1}}
 +  \bar\nu^{-2}\|\check v(t)\|_{\dot B^{-\sigma}_{2,1}}\\
+ {\rm Ma}^{-2}\int_0^t \bigl(\check\nu\|\check a\|^{\ell,\check\nu^{-1}}_{\dot B^{2-\sigma}_{2,1}} +\|\check a\|^{h,\check\nu^{-1}}_{\dot B^{1-\sigma}_{2,1}}
\bigr)
+\int_0^t \|\check v_{y}\|_{\dot B^{1-\sigma}_{2,1}}
\lesssim \cI_{0,\bar\nu}^{-\sigma}.\end{multline}
 Performing a suitable time and space rescaling reduces the proof of \eqref{eq:sigma} to the case ${\rm Ma}=\bar\nu=\bar\eta=1.$ 
 Now, taking advantage of \eqref{eq:linear} with $p=2,$ $s=s'=-\sigma$ and of the definition of $g$ in \eqref{eq:NSClaglin},
  we get
  \begin{multline}\label{eq:finalsigmab}\|a(t)\|_{\dot B^{-\sigma}_{2,1}} 
+ \|a(t)\|_{\dot B^{1-\sigma}_{2,1}}  +  \|v(t)\|_{\dot B^{-\sigma}_{2,1}}  
+\int_0^t \bigl(\|a\|^{\ell,1}_{\dot B^{2-\sigma}_{2,1}} + \|a\|^{h,1}_{\dot B^{1-\sigma}_{2,1}}
+\|v\|_{\dot B^{2-\sigma}_{2,1}}\bigr)\\
\lesssim \|a_0\|_{\dot B^{-\sigma}_{2,1}}
+\|a_0\|_{\dot B^{1-\sigma}_{2,1}} +  \|v_0\|_{\dot B^{-\sigma}_{2,1}}
+\int_0^t \bigl(\|a K(a)\|_{\dot B^{1-\sigma}_{2,1}}+\|L(a)v_y\|_{\dot B^{1-\sigma}_{2,1}}\bigr)\cdotp\end{multline}
Using the stability of the Besov space $\dot B^{1/2}_{2,1}$ by left composition and the 
product law $\dot B^{1/2}_{2,1}\times \dot B^{1-\sigma}_{2,1}\to  \dot B^{1-\sigma}_{2,1}$
that holds true if (and only if)  $-1/2<1-\sigma\leq 1/2,$ we readily get
 $$\begin{aligned} \int_0^t  \bigl(\|a K(a)\|_{\dot B^{1-\sigma}_{2,1}}+\|L(a)v_y\|_{\dot B^{1-\sigma}_{2,1}}\bigr)&\lesssim
 \int_0^t\|a\|_{\dot B^{1/2}_{2,1}}\bigl(\|a\|_{\dot B^{1-\sigma}_{2,1}} +\|v_y\|_{\dot B^{1-\sigma}_{2,1}} \bigr)\\ 
 &\lesssim \|a\|_{L^2_t(\dot B^{1/2}_{2,1})} \|a\|_{L^2_t(\dot B^{1-\sigma}_{2,1})}
 + \|a\|_{L^\infty_t(\dot B^{1/2}_{2,1})} \|v\|_{L^1_t(\dot B^{2-\sigma}_{2,1})}.\end{aligned}
 $$
 Remembering \eqref{eq:final} and the smallness hypothesis \eqref{eq:smalldata}, we discover that the nonlinear terms in 
 \eqref{eq:finalsigmab} may be absorbed by the left-hand side, yielding eventually \eqref{eq:sigma} 
 in the case ${\rm Ma}=\bar\nu=\bar\eta=1,$ and thus in full generality, after reverting to the original variables.  
 \medbreak
 Let us set 
 $$ \wt Q(z):={\rm Ma}^2 Q(z) \andf \wt \nu(z):=\bar\nu^{-1}\nu(z).$$
 To pass to the limit in \eqref{eq:NL3}, the key is to rewrite the equation of $\check a:=\check \eta-\bar\eta$ as
$$\check a_t-{\rm Ma}^{-2} (\wt\nu^{-1}\wt Q)(\bar\eta+\check a) =\check w_y 
\with \check w_y:=\check v_y-{\rm Ma}^{-2}(\wt\nu^{-1}\wt Q)(\check \eta),$$
then to prove that $\check w_y$ converges strongly to zero in a suitable space. 
Before that, let us focus on the limit equation \eqref{eq:limit1}. Since the function $\wt\nu^{-1}\wt Q$ is smooth, this equation
can  be solved locally in time
by means of the Cauchy-Lipschitz theorem. Furthermore, since $(\wt\nu^{-1}\wt Q)'(\bar\eta)<0$ and $\wt Q(0)=0,$    solutions
$\check\theta$ emanating from small perturbations of $\bar\eta$ are  global. 
\medbreak
 Observe that   $\check b:=\check\theta-\bar\eta$ satisfies for some smooth function $k$ vanishing at $0,$ 
 $$\check b_t+{\rm Ma}^{-2}\check b= {\rm Ma}^{-2} \check b\, k(\check b)\andf \check b|_{t=0}=b_0:=\check\theta_0-\bar\eta.$$
Hence, 
routine computations give
$$\begin{aligned}\|\check b\|_{L^\infty_t(\dot B^{1/2}_{2,1})}+{\rm Ma}^{-2}\|\check b\|_{L^1_t(\dot B^{1/2}_{2,1})}& \leq 
\|b_0\|_{\dot B^{1/2}_{2,1}}+ C{\rm Ma}^{-2}\int_0^t \|\check b(t)\|_{\dot B^{1/2}_{2,1}}^2\\
 &\leq \|b_0\|_{\dot B^{1/2}_{2,1}}+ C\|\check b\|_{L^\infty_t(\dot B^{1/2}_{2,1})} \bigl({\rm Ma}^{-2}\|\check b\|_{L^1_t(\dot B^{1/2}_{2,1})}\bigr)\cdotp\end{aligned}
$$
Consequently, if we assume that $\|b_0\|_{\dot B^{1/2}_{2,1}}\ll1$ then, for all $t\geq0,$ we have 
\begin{equation}\label{eq:estimateb}
\|\check b\|_{L^\infty_t(\dot B^{1/2}_{2,1})}+{\rm Ma}^{-2}\|\check b\|_{L^1_t(\dot B^{1/2}_{2,1})} \leq 
2\|b_0\|_{\dot B^{1/2}_{2,1}}.\end{equation}
Now, subtracting the equation of $\check b$ from the one of $\check a$, we discover that
$\da:=\check a-\check b$ satisfies
\begin{equation}\label{eq:da0}\da_t+{\rm Ma}^{-2}\da = \check w_y +{\rm Ma}^{-2} \da\, k(\check b) +{\rm Ma}^{-2} \check a (k(\check a)-k(\check b)).\end{equation}
Let us admit for a while that one can decompose $\check w_y$ into
 \begin{equation}\label{eq:wy}
 \check w_y=A+B\with A\in L^1(\R_+;\dot B^{1/2}_{2,1})\andf B\in L^2(\R_+;\dot B^{1/2}_{2,1}).\end{equation}
 Then,  it is easy to get from \eqref{eq:da0} that
 \begin{multline}\label{eq:da1}\|\da\|_{L^\infty_t(\dot B^{1/2}_{2,1})}+{\rm Ma}^{-1}\|\da\|_{L^2_t(\dot B^{1/2}_{2,1})} \leq \|\da_0\|_{\dot B^{1/2}_{2,1}}
+\|A\|_{L^1_t(\dot B^{1/2}_{2,1})} + {\rm Ma}\|B\|_{L^2_t(\dot B^{1/2}_{2,1})}\\
+{\rm Ma}^{-2}\| \da\, k(\check b) +\check a (k(\check a)-k(\check b))\|_{L^1_t(\dot B^{1/2}_{2,1})}.\end{multline}
The second line is  harmless. Indeed, 
by \eqref{eq:compo} and product laws, we have
$$\begin{aligned}\|\da \, k(\check b)\|_{L_t^1(\dot B^{1/2}_{2,1})}&\lesssim \|\check b\|_{L^2_t(\dot B^{1/2}_{2,1})} \|\da\|_{L^2_t(\dot B^{1/2}_{2,1})},\\
\|\check a (k(\check a)-k(\check b))\|_{L^1_t(\dot B^{1/2}_{2,1})}&\lesssim 
 \|\check a\|_{L^2_t(\dot B^{1/2}_{2,1})} \|\da\|_{L^2_t(\dot B^{1/2}_{2,1})}.\end{aligned}$$
 As $\check a$ and $\check b$ are small  in $L^2(\R_+;\dot B^{1/2}_{2,1})$  
 (remember \eqref{eq:finalbis} and \eqref{eq:estimateb}), these terms may be absorbed by the left-hand side of \eqref{eq:da1}, yielding eventually
 \begin{equation}\label{est:da}
 \|\da\|_{L^\infty_t(\dot B^{1/2}_{2,1})}+{\rm Ma}^{-1}\|\da\|_{L^2_t(\dot B^{1/2}_{2,1})} \lesssim  \|\da_0\|_{\dot B^{1/2}_{2,1}}
+\|A\|_{L^1_t(\dot B^{1/2}_{2,1})} + {\rm Ma}\|B\|_{L^2_t(\dot B^{1/2}_{2,1})}.\end{equation}
 There only remains to justify \eqref{eq:wy} with $A$ and $B$ tending to $0$ in the desired spaces
 when $\bar\nu$ goes to $\infty.$
  To do so,  we set
$$
 A:=\check w_y^{h,\alpha} + \check v_y^{\ell,\alpha}\andf B:=-{\rm Ma}^{-2}\bigl((\wt\nu^{-1}\wt Q)(\check \eta)\bigr)^{\ell,\alpha}$$
where the parameter $\alpha$ is chosen so that $\check\nu^{-2}\ll \alpha\ll 1.$
\smallbreak
We  shall use repeatedly the obvious fact that 
\begin{equation}\label{eq:lfhfbis}
\|z\|_{\dot B^\sigma_{2,1}}^{\ell,\alpha}\lesssim \alpha^{\sigma-\sigma'}\|z\|_{\dot B^{\sigma'}_{2,1}}^{\ell,\alpha}\andf
\|z\|_{\dot B^{\sigma'}_{2,1}}^{h,\alpha}\lesssim \alpha^{\sigma'-\sigma}\|z\|_{\dot B^{\sigma}_{2,1}}^{h,\alpha},\qquad \sigma'\leq\sigma.
\end{equation}
In order to bound $B,$ it suffices to use that  $-(\wt \nu^{-1}\wt Q)(\check \eta)= \check a + f(\check a) \check a$ for some smooth function $f$
vanishing at $0.$ Hence,  due to product and composition laws, 
$$
\|(\wt \nu^{-1}\wt Q)(\check \eta)\|_{\dot B^{1-\sigma}_{2,1}}\lesssim 
\|\check a\|_{\dot B^{1-\sigma}_{2,1}}\bigl(1+ \|\check a\|_{\dot B^{1/2}_{2,1}}\bigr)\cdotp
$$
Therefore, using  \eqref{eq:finalsigmab} and  \eqref{eq:lfhfbis}, we arrive at
 \begin{equation}\label{eq:B}  {\rm Ma}\|B\|_{L^2_t(\dot B^{1/2}_{2,1})}\lesssim 
  \alpha^{\sigma-1/2}{\rm Ma}^{-1}\|(\wt \nu^{-1}\wt Q)(\check a)\|_{L^2_t(\dot B^{1-\sigma}_{2,1})}\lesssim  \alpha^{\sigma-1/2}\cI_{0,\bar\nu}^{-\sigma}.
  \end{equation}
Let us turn to the study of $A.$ 
The low frequency part is easy: from  \eqref{eq:finalsigmab} and \eqref{eq:lfhfbis}, we have 
\begin{equation}\label{eq:A1}
\| \check v_y\|^{\ell,\alpha}_{L^1_t(\dot B^{1/2}_{2,1})} \leq \alpha^{\sigma-1/2}\|\check v_y\|_{L^1_t(\dot B^{1-\sigma}_{2,1})}^{\ell,\alpha}
\lesssim \alpha^{\sigma-1/2} \cI_{0,\bar\nu}^{-\sigma}.\end{equation}
There only remains to  bound the high frequencies of $\check w_y.$ To do so, we observe that
 \begin{equation}\label{eq:NL4}
(\check w_y)_t - \bar\nu^{2}(\wt\nu \check w_y)_{yy}= -{\rm Ma}^{-2}(\wt\nu^{-1}\wt Q)'(\check \eta)\check v_y,
 \end{equation}
which can be rewritten
 $$(\check w_y)_t - \bar\nu^{2}(\check w_y)_{yy}= \bar\nu^2(k_1(\check a)\check w_y)_{yy}
 +  {\rm Ma}^{-2}(1+k_2(\check a))\check v_y,$$
where the functions $k_1,\, k_2$  are  smooth and vanish  at zero.
\medbreak
By parabolic maximal regularity (restricted to high frequencies of $w$), we readily have
 \begin{equation}\label{eq:da2}
\bar\nu^2 \|\check w_y\|^{h,\alpha}_{L_t^1(\dot B^{1/2}_{2,1})}\lesssim\|\check w_{y,0}\|_{\dot B^{-3/2}_{2,1}}^{h,\alpha}
 + \bar\nu^{2} \|k_1(\check a)\check w_y\|^{h,\alpha}_{L_t^1(\dot B^{1/2}_{2,1})}+ {\rm Ma}^{-2}\|(1+k_2(\check a))\check v_y\|^{h,\alpha}_{L_t^1(\dot B^{-3/2}_{2,1})}.\end{equation}
Due to \eqref{eq:lfhfbis}, we have 
  $$ \|\check w_{y,0}\|_{\dot B^{-3/2}_{2,1}}^{h,\alpha}\lesssim
  \bar\nu\|v_{0}\|_{\dot B^{-1/2}_{2,1}}^{h,\alpha}+{\rm Ma}^{-2}\alpha^{-2}\|\check a_0\|_{\dot B^{1/2}_{2,1}}.$$
Next, thanks to the usual product laws, 
$$\begin{aligned}
 \|k_1(\check a)\check w_y\|^{h,\alpha}_{L_t^1(\dot B^{1/2}_{2,1})}
& \lesssim  \|k_1(\check a)\check w_y^{h,\alpha}\|_{L_t^1(\dot B^{1/2}_{2,1})}
+ \|k_1(\check a)\check v_y^{\ell,\alpha}\|_{L_t^1(\dot B^{1/2}_{2,1})}+\|k_1(\check a)B\|_{L_t^1(\dot B^{1/2}_{2,1})}\\
&\lesssim \|\check a\|_{L^\infty_t(\dot B^{1/2}_{2,1})} \bigl(\|\check w_y\|^{h,\alpha}_{L_t^1(\dot B^{1/2}_{2,1})}+
 \|\check v_y^{\ell,\alpha}\|_{L_t^1(\dot B^{1/2}_{2,1})}\bigl)+ \|\check a\|_{L_t^2(\dot B^{1/2}_{2,1})}\|B\|_{L_t^2(\dot B^{1/2}_{2,1})}.
\end{aligned}
$$
Thanks to the smallness of $\check a$ in the critical regularity space, the first term may be absorbed by the left-hand side
of \eqref{eq:da2}.  As for the second and third terms, they may be bounded thanks to \eqref{eq:B} and \eqref{eq:A1}.
  Next,  we write 
$$\begin{aligned}\|(1+k_2(\check a))\check v_y\|_{L_t^1(\dot B^{-3/2}_{2,1})}^{h,\alpha}
&\lesssim \alpha^{-2} \|(1+k_2(\check a))\check v_y\|_{L_t^1(\dot B^{1/2}_{2,1})}\\
&\lesssim   \alpha^{-2}\bigl(1+\|\check a\|_{L^\infty_t(\dot B^{1/2}_{2,1})}\bigr)\|\check v_y\|_{L_t^1(\dot B^{1/2}_{2,1})}
\end{aligned}$$
and use \eqref{eq:finalbis}.  Back to \eqref{eq:da2}, we end up with 
  $$ \|\check w_y\|^{h,\alpha}_{L_t^1(\dot B^{1/2}_{2,1})}\lesssim 
  \bar\nu^{-1}\|v_0\|_{\dot B^{-1/2}_{2,1}}+\alpha^{-2}\check \nu^{-2} \cI_{0,\nu}^{-1/2} +\alpha^{\sigma-1/2} 
 \cI_{0,\bar\nu}^{-1/2}\cI_{0,\bar\nu}^{-\sigma}.$$
  Consequently, if we assume that $I_{0,\bar\nu}^{-\sigma}\ll\check\nu^{\sigma-1/2}$ and choose  $\alpha$ such that
  $$   \alpha^{-2}\check \nu^{-2}= \alpha^{\sigma-1/2} I_{0,\bar\nu}^{-\sigma},$$
  and use also \eqref{eq:B} and \eqref{eq:A1}, we conclude that $\check w_y= A+B$ with 
   $$
   \|A\|_{L_t^1(\dot B^{1/2}_{2,1})}+{\rm Ma}      \|B\|_{L_t^2(\dot B^{1/2}_{2,1})}  \lesssim 
   \check \nu^{-\varsigma}\bigl(\cI_{0,\bar\nu}^{-\sigma}\bigr)^{\frac4{3+2\sigma}} +\bar\nu^{-1} \|v_0\|_{\dot B^{-1/2}_{2,1}}
   \with \varsigma:=   \frac{2\sigma-1}{\sigma+3/2}\cdotp$$
Reverting to \eqref{est:da} yields
 \begin{equation}\label{eq:dafinal} \|\da\|_{L^\infty_t(\dot B^{1/2}_{2,1})}+{\rm Ma}^{-1}\|\da\|_{L^2_t(\dot B^{1/2}_{2,1})} \lesssim 
\|\da_0\|_{\dot B^{1/2}_{2,1}} +  \check \nu^{-\varsigma}\bigl(\bar \cI_{0,\bar\nu}^{-\sigma}\bigr)^{\frac4{3+2\sigma}} +\bar\nu^{-1} \|v_0\|_{\dot B^{-1/2}_{2,1}}.\end{equation}
In the case of initial data independent of $\nu,$ we have $\da_0=0,$ whence  the uniform convergence
of $\check\eta$ to $\check\theta$  (on $\R_+\times\R$),  with the rate $\bar\nu^{-\varsigma}.$ \qed

 \section{Appendix}
 
 Here we motivate the scaling that we used for 
 the diffusive limit, at the linear level, by
 computing explicit formulae in the Fourier space for the solution of the linearized compressible
 Navier-Stokes equations.
 In passing, we point out the so-called `overdamping phenomenon' that may be observed 
 if we keep the original scaling and let the viscosity tend to infinity. 
 \medbreak
 As dimension $1$ does not play any role here, we consider the general multi-dimensional linearized 
 compressible Navier-Stokes equations,   namely
 \begin{equation}\label{eq:linearnsc}
 \left\{\begin{array}{l}
 a_t+\alpha \,\div z=0,\\[1ex]
 z_t-\lambda \Delta z - (\lambda+\lambda')\nabla\div z+\beta\,\nabla a=0.\end{array}\right.
 \end{equation}
 We make the stability  hypothesis\footnote{The standard linearized equations correspond 
 to $\alpha>0$ and $\beta>0.$ However, in the mass Lagrangian coordinates system here considered, 
 these two coefficients are \emph{negative}.}:
 $$ {\rm sgn}\,\alpha={\rm sgn}\,\beta\not=0,\quad \lambda>0\andf\mu:=2\lambda+\lambda'>0.$$
The divergence free part ${\mathcal P}z$ of $z$ just satisfies the heat equation with diffusion $\lambda,$
 and  the coupling between $a$ and $u:=\Lambda^{-1}\div z$ (with $\Lambda^s:=(-\Delta)^{s/2}$)
 is governed by:
  \begin{equation}\label{eq:linearnscsimp}
 \left\{\begin{array}{l}
 a_t+\alpha \Lambda u=0,\\[1ex]
 u_t-\mu \Delta u - \beta\Lambda a=0.\end{array}\right.
 \end{equation}
 Performing the space and time rescaling $$
a(t,x)=\check a\Bigl(\frac{\alpha\beta}\mu t,\frac{\sqrt{\alpha\beta}}\mu  x\Bigr),\quad
u(t,x)=\sqrt{\frac\beta\alpha}\check u\Bigl(\frac{\alpha\beta}\mu t,\frac{\sqrt{\alpha\beta}}\mu x\Bigr)$$
we discover that 
 $(\check a,\check u)$ satisfies \eqref{eq:linearnscsimp} with  $\alpha=\beta=\mu=1.$ 
Then, taking advantage of the formula computed in e.g. \cite{CD}, and scaling back, we end up with:
\begin{itemize}
\item  For $\mu|\xi|>2\sqrt{\alpha\beta}$~: 
\begin{align*}
\wh a(t,\xi)&= e^{t\lambda^{-}(\xi)}\biggl(\frac12\Bigl(1\!+\!\frac1{R(\xi)}\Bigr)\wh a_0(\xi)
-\frac\alpha{\mu|\xi|R(\xi)}\wh u_0(\xi)\biggr)\\\hspace{2cm}&+
e^{t\lambda^{+}(\xi)}\biggl(\frac12\Bigl(1\!-\!\frac1{R(\xi)}\Bigr)\wh a_0(\xi)
+\frac\alpha{\mu|\xi|R(\xi)}\wh u_0(\xi)\biggr),\\[1ex]
\wh u(t,\xi)&= e^{t\lambda^{-}(\xi)}\biggl(\frac\beta{\mu|\xi|R(\xi)}\wh a_0(\xi)+
\frac12\Bigl(1\!-\!\frac1{R(\xi)}\Bigr)\wh u_0(\xi)\biggr)\\\hspace{2cm}&+
 e^{t\lambda^{+}(\xi)}\biggl(-\frac\beta{\mu|\xi|R(\xi)}\wh a_0(\xi)+
\frac12\Bigl(1\!+\!\frac1{R(\xi)}\Bigr)\wh u_0(\xi)\biggr)
\end{align*}
 with $R(\xi):=\Bigl(1-\frac{4\alpha\beta}{\mu^2|\xi|^2}\Bigr)^{1/2}$ and 
 $\lambda^\pm(\xi):=-\frac{\mu|\xi|^2}{2}(1\pm R(\xi))$. \smallbreak
\item  
 For $\mu|\xi|<2\sqrt{\alpha\beta}$~: 
\begin{align*}
\wh a(t,\xi)&= e^{t\lambda^{-}(\xi)}\biggl(\frac12\Bigl(1\!+\!\frac i{R(\xi)}\Bigr)\wh a_0(\xi)
-\frac{i\alpha}{\mu|\xi|R(\xi)}\wh u_0(\xi)\biggr)\\\hspace{2cm}&+
e^{t\lambda^{+}(\xi)}\biggl(\frac12\Bigl(1\!-\!\frac i{R(\xi)}\Bigr)\wh a_0(\xi)
+\frac{i\alpha}{\mu|\xi|R(\xi)}\wh u_0(\xi)\biggr),\\[1ex]
\wh u(t,\xi)&= e^{t\lambda^{-}(\xi)}\biggl(\frac{i\beta}{\mu|\xi|R(\xi)}\wh a_0(\xi)+
\frac12\Bigl(1\!-\!\frac i{R(\xi)}\Bigr)\wh u_0(\xi)\biggr)\\\hspace{2cm}&+
 e^{t\lambda^{+}(\xi)}\biggl(-\frac{i\beta}{\mu|\xi|R(\xi)}\wh a_0(\xi)+
\frac12\Bigl(1\!+\!\frac i{R(\xi)}\Bigr)\wh u_0(\xi)\biggr)
\end{align*}
 with $R(\xi):=\Bigl(\frac{4\alpha\beta}{\mu^2|\xi|^2}-1\Bigr)^{1/2}$ and 
 $\lambda^\pm(\xi):=-\frac{\mu|\xi|^2}{2}(1\pm i R(\xi))$. 
 \end{itemize}
 As the potential part $\cQ z$ of $z$ satisfies $\cF(\cQ z)(\xi)=-i\frac{\xi \wh u(\xi)}{|\xi|},$ we deduce that for $\mu|\xi|>2\sqrt{\alpha\beta},$ 
 \begin{align*}\wh a(t,\xi)&= e^{t\lambda^{-}(\xi)}\biggl(\frac12\Bigl(1\!+\!\frac1{R(\xi)}\Bigr)\wh a_0(\xi)
-\frac{i\alpha}{\mu|\xi|^2R(\xi)} 
\xi\cdot\wh z_0(\xi)\biggr)
\\\hspace{2cm}&+
e^{t\lambda^{+}(\xi)}\biggl(\frac12\Bigl(1\!-\!\frac1{R(\xi)}\Bigr)\wh a_0(\xi)
+\frac{i\alpha}{\mu|\xi|^2R(\xi)}\xi\cdot\wh z_0(\xi)\biggr),\\[1ex]
\wh{\cQ z}(t,\xi)&= e^{t\lambda^{-}(\xi)}\biggl(-\frac{i\beta}{\mu|\xi|^2R(\xi)}\xi\wh a_0(\xi)+
\frac12\bigl(1\!-\!\frac1{R(\xi)}\bigr)\wh{\cQ z_0}(\xi)\biggr)\\\hspace{2cm}&+
 e^{t\lambda^{+}(\xi)}\biggl(\frac{i\beta}{\mu|\xi|^2R(\xi)}\xi\wh a_0(\xi)+
\frac12\Bigl(1\!+\!\frac1{R(\xi)}\Bigr)\wh{\cQ z_0}(\xi)\biggr),
\end{align*}
and a similar formula for  $\mu|\xi|<2\sqrt{\alpha\beta}.$
\medbreak 
Under the original scaling of \eqref{eq:NL2}, for any fixed frequency $\xi\not=0$ and 
for $\bar\nu\to\infty,$ we have 
$$
\lambda^+(\xi)\simeq -\bar\nu\xi^2,\quad 
\lambda^-(\xi)\simeq -\bar\nu^{-1}{\rm Ma}^{-2} \andf \frac12\biggl(1-\frac1{R(\xi)}\biggr)\simeq\frac{1}{\check\nu^2\xi^2}\,,
$$
whence 
$$\begin{aligned}
\wh a(t,\xi)&\simeq e^{-\frac{t}{\bar\nu{\rm Ma}^2}}\biggl(\wh a_0(\xi)+\frac{i\wh v_0(\xi)}{\bar\nu\xi}\biggr)
\andf\\
\wh v(t,\xi)&\simeq e^{-t\bar\nu\xi^2}\wh v_0(\xi)+ e^{-\frac{t}{\bar\nu{\rm Ma}^2}}\biggl(\frac{i\wh a_0(\xi)}{\bar\nu{\rm Ma}^2\xi}
+\frac{\wh v_0(\xi)}{\check\nu^2\xi^2}\biggr)
\cdotp\end{aligned}$$
Hence, we  get the trivial asymptotics that $a\to a_0$ and $v\to 0$ when $\bar\nu$ goes to infinity.
\medbreak
In contrast, using the diffusive scaling of \eqref{eq:NL3}, that is,  taking $\alpha=-1,$  $\beta=-{\rm Ma}^{-2}\bar\nu^2$ and $\mu=\bar\nu^2$ in the above relations yields 
 for $|\xi|>2\check\nu^{-1}$:
 \begin{align*}
\wh a(t,\xi)&= e^{t\lambda^{-}(\xi)}\biggl(\frac12\Bigl(1\!+\!\frac1{R(\xi)}\Bigr)\wh a_0(\xi)
+\frac i{\bar\nu^2 \xi R(\xi)}  \wh v_0(\xi)\biggr)\\\hspace{2cm}&+
e^{t\lambda^{+}(\xi)}\biggl(\frac12\Bigl(1\!-\!\frac1{R(\xi)}\Bigr)\wh a_0(\xi)-\frac i{\bar\nu^2 \xi R(\xi)}\wh v_0(\xi)\biggr),\\[1ex]
\wh v(t,\xi)&= e^{t\lambda^{-}(\xi)}\biggl(\frac i{{\rm Ma}^2\xi R(\xi)} \wh a_0(\xi)+
\frac12\Bigl(1\!-\!\frac1{R(\xi)}\Bigr)\wh v_0(\xi)\biggr)\\\hspace{2cm}&+
 e^{t\lambda^{+}(\xi)}\biggl(- \frac i{{\rm Ma}^2 \xi R(\xi)}\wh a_0(\xi)+
\frac12\Bigl(1\!+\!\frac1{R(\xi)}\Bigr)\wh v_0(\xi)\biggr),
\end{align*} 
with $R(\xi):=\Bigl(1-\frac{4}{\check\nu^2\xi^2}\Bigr)^{1/2}$ and 
 $\lambda^\pm(\xi):=-\frac{\bar\nu^2\xi^2}{2}(1\pm R(\xi)).$
\medbreak
We deduce that  $w_y:=v_y+{\rm Ma}^{-2}a$ is given by the formula
\begin{equation*}
\wh{w_y}(t,\xi)= \frac12\biggl(\Bigl(1\!-\!\frac1{R(\xi)}\Bigr) e^{t\lambda^-(\xi)} 
+\Bigl(1\!+\!\frac1{R(\xi)}\Bigr) e^{t\lambda^+(\xi)} \biggr)\wh{w_{0,y}}(\xi)
+\biggl(\frac{e^{t\lambda^-(\xi)}-e^{t\lambda^+(\xi)}}{R(\xi)\check\nu^2\xi^2} \biggr)\wh{v_{0,y}}(\xi).
\end{equation*}
Consequently, for fixed $\xi\not=0$ and $\bar\nu$ going to zero, we have 
$$\wh a(t,\xi)=e^{-t} \wh a_0(\xi)\andf  
\wh w_y(t,\xi)\simeq \biggl(e^{-\bar\nu^2t|\xi|^2}-\frac{e^{-t}}{\check\nu^2|\xi|^2}\wh w_{0,y}\biggr)
+\frac{e^{-t}}{\check\nu^2|\xi|^2}\,\wh v_{0,y}(\xi).$$

Hence $a$ converges to the solution of the linearized equation \eqref{eq:limit1}, while
$\wh w$ goes to zero, with a rate that can be quantified in terms of norms and (negative) powers of $\bar\nu.$

{\medbreak\noindent{\bf Acknowledgments.}  
The  author is grateful to the anonymous referee, 
whose suggestions helped clarify certain passages in the demonstrations.}

\end{document}